\documentclass[12pt,a4paper,reqno]{amsart}

\usepackage{xcolor}
\usepackage[all,2cell,arrow,color]{xy}
\newdir{ >}{{}*!/-10pt/@{>}}

\usepackage[english]{babel}
\usepackage{mathtools}
\usepackage{stmaryrd}
\usepackage{tensor}
\numberwithin{equation}{section}
\usepackage{graphicx} 
\usepackage{rotating}
\usepackage{pdflscape}
\usepackage[pagebackref,colorlinks,urlcolor=darkgreen,citecolor=darkgreen,linkcolor=darkgreen]{hyperref}
\usepackage{enumerate,xspace}
\usepackage{xcolor}
\usepackage{xcolor}

\definecolor{darkgreen}{rgb}{0,0.45,0}

  \newtheorem{proposition}{Proposition}[section]
  \newtheorem{lemma}[proposition]{Lemma}
  \newtheorem{corollary}[proposition]{Corollary}
  \newtheorem{theorem}[proposition]{Theorem}

  \theoremstyle{definition}
  \newtheorem{definition}[proposition]{Definition}
  \newtheorem{example}[proposition]{Example}

  \newtheorem{assumption}[proposition]{Assumption}
  
\theoremstyle{remark}

  \newcounter{c}
  \renewcommand{\[}{\setcounter{c}{1}$$}
  \newcommand{\etyk}[1]{\vspace{-7.4mm}$$\begin{equation}\Label{#1}
  \addtocounter{c}{1}}
  \renewcommand{\]}{\ifnum \value{c}=1 $$\else \end{equation}\fi}
\makeatletter
\newcommand*{\inlineequation}[2][]{%
  \begingroup
    \refstepcounter{equation}%
    \ifx\\#1\\%
    \else
      \label{#1}%
    \fi
    \relpenalty=10000 %
    \binoppenalty=10000 %
    \ensuremath{%
      #2%
    }%
    ~\@eqnnum
  \endgroup
}

\makeatletter
\def\@settitle{\begin{center}%
  \baselineskip14\p@\relax
  \bfseries
  \uppercasenonmath\@title
  \@title
  \ifx\@subtitle\@empty\else
     \\[1ex]\uppercasenonmath\@subtitle
     \footnotesize\mdseries\@subtitle
  \fi
  \end{center}%
}
\def\subtitle#1{\gdef\@subtitle{#1}}
\def\@subtitle{}
\makeatother

\newcommand{\coten}[1]{\raisebox{-7pt}{\ensuremath{\stackrel{\displaystyle  \Box}{\scriptstyle { #1}}}}}
\newcommand{\morcoten}{\scalebox{.7}{\ensuremath{\,\Box \,}}}
\newcommand{\diagcoten}{\scalebox{.55}{\ensuremath \Box}}
\newcommand{\expcoten}[1]{\scalebox{.55}{
{\raisebox{-8pt}{\ensuremath{\stackrel{\displaystyle  \Box}{\scriptstyle { #1}}}}}}}

\begin{document}

\title[Crossed modules of monoids III.]{Crossed modules of monoids III.}
\subtitle{Simplicial monoids of Moore length 1}

\author{Gabriella B\"ohm} 
\address{Wigner Research Centre for Physics, H-1525 Budapest 114,
P.O.B.\ 49, Hungary}
\email{bohm.gabriella@wigner.mta.hu}
\date{March 2018}
  
\begin{abstract}
This is the last part of a series of three strongly related papers in which
three equivalent structures are studied:
\begin{itemize}
\item[-] internal categories in categories of monoids; defined in terms of
pullbacks relative to a chosen class of spans
\item[-] crossed modules of monoids relative to this class of spans 
\item[-] simplicial monoids of so-called Moore length 1 relative to this class
  of spans. 
\end{itemize}
The most important examples of monoids that are covered are small categories
(treated as monoids in categories of spans) and bimonoids in symmetric
monoidal categories (regarded as monoids in categories of comonoids). 
In this third part relative simplicial monoids are analyzed. 
Their Moore length is introduced and the equivalence is proven between
relative simplicial monoids of Moore length 1, and relative categories of
monoids in Part I. This equivalence is obtained in one direction by truncating
a simplicial monoid at the first two degrees; and in the other direction by
taking the simplicial nerve of a relative category. 
\end{abstract}
  
\maketitle


\section*{Introduction} \label{sec:intro}

Whitehead's crossed modules of groups \cite{Whitehead} received a lot of
attention because of their appearance in many different contexts; see the
review articles \cite{Porter:Menagerie,Porter:HQFT,Paoli} and the references
in them. In many of the applications they did not appear in their original
form; but in the disguise of the equivalent notion of strict 2-group (see
\cite{BrownSpencer,Duskin,Janelidze} for proofs of the equivalence).

Groups can be regarded as (distinguished) monoids in the Cartesian monoidal
category of sets.
In our antecedent papers \cite{Bohm:Xmod_I} and \cite{Bohm:Xmod_II} we worked
out the notion of crossed module of monoids in more general, not necessarily
Cartesian monoidal categories, relative to a chosen suitable class of spans.
The main examples described by the theory are crossed modules of groupoids in
\cite{BrownIcen} and crossed modules of Hopf monoids in
\cite{Aguiar,Villanueva,Majid,FariaMartins}. 

The aim of this article is to extend to this level of generality a third
equivalent description of crossed modules of groups: as simplicial groups
whose Moore complex has length 1.

We consider simplicial monoids; that is, functors from the opposite of the
simplicial category $\Delta$ to the category of monoids in some monoidal
category $\mathsf C$. Assuming the existence of certain relative pullbacks
(cf. \cite{Bohm:Xmod_I}), in Section \ref{sec:Moore} we associate to a
simplicial monoid  a sequence of morphisms which yields a chain complex
whenever the monoidal unit of $\mathsf C$ is a terminal (hence zero) object in the category of monoids in $\mathsf C$. It can be seen
as a generalization of the Moore complex of a simplicial group. We also give a
meaning to its length and study the consequences of its having some finite
length $l$. 

Section \ref{sec:iso} contains some technical material about the
invertibility of a certain canonical morphism playing essential role in the
theory. 

The main result can be found in Section \ref{sec:equivalence} where we  
prove an equivalence between the category of relative categories in the
category of monoids in $\mathsf C$ (cf. \cite{Bohm:Xmod_I}) and the category
of relative simplicial monoids in $\mathsf C$ whose Moore length is 1. 
The functors establishing the equivalence carry transparent meanings. For obtaining a relative category from a simplicial monoid we truncate it at the
first two degrees. In the opposite direction, a simplicial monoid is obtained
from a relative category of monoids as the (relative) simplicial nerve. 

In the particular monoidal category of spans over a given set, we obtain an
equivalence between the categories of certain double categories, and of certain
simplicial categories. These equivalent categories contain, as equivalent full
subcategories, the category of 2-groupoids on one hand, and the category of
simplicial groupoids with length 1 Moore complex on the other hand.

In the particular monoidal category of comonoids in some symmetric monoidal
category $\mathsf M$, we obtain an equivalence between the categories of
certain category objects in the category of bimonoids in $\mathsf M$, and of
certain simplicial bimonoids. These equivalent categories contain, as
equivalent full subcategories, the category of internal categories in the
category of cocommutative Hopf monoids in $\mathsf M$ on one hand, and the
category of simplicial cocommutative Hopf monoids of  length 1 Moore complex on the
other hand. This includes, in particular, the equivalence in \cite{Emir}
between the category objects in the category of cocommutative Hopf algebras
over a field, and the category of simplicial cocommutative Hopf algebras whose
Moore complex has length 1.   

\subsection*{Acknowledgement} 
The author's interest in the subject was triggered by the excellent workshop
{\em `Modelling Topological Phases of Matter -- TQFT, HQFT, premodular and
  higher categories, Yetter-Drinfeld and crossed modules in disguise'} in
Leeds UK, 5-8 July 2016. It is a pleasure to thank the organizers,
Zolt\'an K\'ad\'ar, Jo\~ao Faria Martins, Marcos Cal\c{c}ada and Paul Martin
for the experience and a generous invitation.
Financial support by the Hungarian National Research, Development and
Innovation Office – NKFIH (grant K124138) is gratefully acknowledged.


\section{The Moore length}
\label{sec:Moore}

In the category of groups (and in more general semi-Abelian categories
\cite{TVdL}), a chain complex --- the so called {\em Moore complex} --- can be 
associated to any simplicial object. The full subcategory of those simplicial
objects whose Moore complex has length 1, turns out to be equivalent to the
category of crossed modules and to the category of internal categories, see
e.g. \cite{Porter:Menagerie}.

However, as in \cite{Bohm:Xmod_I} and \cite{Bohm:Xmod_II}, here we work in categories (of monoids in some monoidal category $\mathsf
C$) where the existence of zero objects is not assumed. So the notion of chain
complex is not available. We shall see below, however, that when certain
relative pullbacks exist, we can still associate to a simplicial object $S$ a
sequence of composable morphisms (reproducing the Moore complex whenever the
monoidal unit of $\mathsf C$ is terminal in the category of monoids in
$\mathsf C$). Although the chain condition can not be formulated at this level
of generality, there is a natural way to define the length of this sequence
that we call the {\em Moore length} of the simplicial monoid $S$. 

The construction of this sequence of composable morphisms is based on the
following assumption.  

\begin{assumption} \label{ass:Moore}
Let $\mathcal S$ be a monoidal admissible class of spans --- in the sense of \cite[Definitions 2.1 and 2.5]{Bohm:Xmod_I} --- in some monoidal category $\mathsf C$. For any simplicial object 
\begin{equation} \label{eq:simp_obj}
\xymatrix@C=30pt{
S_0 \ar[r]|-{\,\sigma_0\,} &
\ar@<-5pt>[l]_-{\partial_0}\ar@<5pt>[l]^-{\partial_1} S_1
\ar@<7pt>[r]|-{\,\sigma_0 \,} \ar@<-7pt>[r]|-{\,\sigma_1 \,} &
\ar[l]|-{\,\partial_1 \,}
\ar@<-12pt>[l]_-{\,\partial_0 \,}\ar@<12pt>[l]^-{\,\partial_2 \,}
S_2 \dots S_{n-1} 
\ar@<14pt>[r]|-{\,\sigma_0 \,} \ar@<-17pt>[r]|-{\,\sigma_{n-1} \,} &
\ar@<-21pt>[l]_-{\,\partial_0 \,} \ar@<-7pt>[l]|-{\,\partial_1 \,}^-\vdots
\ar@<24pt>[l]^-{\,\partial_n \,}
S_n \dots}
\end{equation}
in the category of monoids in $\mathsf C$, consider the following successive
assumptions.  
\begin{itemize}
\item[{(1)}] Assume that for any positive integer $n$ there exists the
  $\mathcal S$-relative pullback --- in the sense of \cite[Definition 3.1]{Bohm:Xmod_I} --- $S_n^{(1)}$ in 
$$
\xymatrix@C=10pt@R=10pt{
S_{n+1}^{(1)} \ar@/^1.2pc/[rrrd]^-{p_I} \ar[ddd]_-{p_{S_{n+1}}}
\ar@{-->}[rd]^(.55){\partial_k^{(1)}} \\
& S_n^{(1)}:=S_n \coten {S_{n-1}} I \ar[rr]^-{p_I} \ar[dd]_-{p_{S_n}} &&
I \ar[dd]^-u \\
\\
S_{n+1} \ar[r]_-{\partial_k} &
S_n \ar[rr]_-{\partial_n} &&
S_{n-1}.}
$$
Since $\partial_k$ is compatible with the units of the monoids $S_n$ and
$S_{n-1}$; and by the simplicial relation
$\partial_n.\partial_k=\partial_k.\partial_{n+1}$ for any $0\leq k \leq n$, we 
may apply \cite[Proposition 3.5~(1)]{Bohm:Xmod_I} to conclude on the
existence of the unique morphism $\partial_k^{(1)}:=\partial_k \morcoten 1$
rendering commutative the above diagram. 
\item[{(2)}] In addition to the assumption in item (1) above, assume that for
  all $n\geq 2$ there exists the $\mathcal S$-relative pullback $S_n^{(2)}$ in 
$$
\xymatrix@C=10pt@R=10pt{
S_{n+1}^{(2)} \ar@/^1.2pc/[rrrd]^-{p_I} \ar[ddd]_-{p_{S_{n+1}^{(1)}}}
\ar@{-->}[rd]^(.55){\partial_k^{(2)}} \\
& S_n^{(2)}:=S_n^{(1)} \coten {S_{n-1}^{(1)}} I \ar[rr]^-{p_I} 
\ar[dd]_-{p_{S_n}^{(1)}} &&
I \ar[dd]^-u \\
\\
S_{n+1}^{(1)} \ar[r]_-{\partial_k^{(1)}} &
S_n^{(1)} \ar[rr]_-{\partial_{n-1}^{(1)}} &&
S_{n-1}^{(1)}.}
$$
Since $\partial_k^{(1)}$ is compatible with the units of the monoids
$S_n^{(1)}$ and $S_{n-1}^{(1)}$; and by the simplicial relation
$\partial_{n-1}.\partial_k=\partial_k.\partial_n$ for any $0\leq k \leq n-1$, we 
may apply \cite[Proposition 3.5~(1)]{Bohm:Xmod_I} to conclude on the
existence of the unique morphism $\partial_k^{(2)}:=\partial_k^{(1)} \morcoten 1$
rendering commutative the above diagram. 
\item[{\vdots}]
\item[{(l)}] In addition to the assumptions in all items (1)\dots (l-1) above,
  assume that for all $n\geq l$ there exists the $\mathcal S$-relative
  pullback $S_n^{(l)}$ in  
$$
\xymatrix@C=10pt@R=10pt{
S_{n+1}^{(l)} \ar@/^1.2pc/[rrrd]^-{p_I} \ar[ddd]_-{p_{S_{n+1}^{(l-1)}}}
\ar@{-->}[rd]^(.55){\partial_k^{(l)}} \\
& S_n^{(l)}:=S_n^{(l-1)} \coten {S_{n-1}^{(l-1)}} I \ar[rr]^-{p_I} 
\ar[dd]_-{p_{S_n}^{(l-1)}} &&
I \ar[dd]^-u \\
\\
S_{n+1}^{(l-1)} \ar[r]_-{\partial_k^{(l-1)}} &
S_n^{(l-1)} \ar[rr]_-{\partial_{n-(l-1)}^{(l-1)}} &&
S_{n-1}^{(l-1)}.}
$$
Since $\partial_k^{(l-1)}$ is compatible with the units of the monoids
$S_n^{(l-1)}$ and $S_{n-1}^{(l-1)}$; and by the simplicial relation
$\partial_{n-(l-1)}.\partial_k=\partial_k.\partial_{n+1-(l-1)}$ for any $0\leq
k \leq n-(l-1)$, we may apply \cite[Proposition 3.5~(1)]{Bohm:Xmod_I}
to conclude on the existence of the unique morphism
$\partial_k^{(l)}:=\partial_k^{(l-1)} \morcoten 1$ 
rendering commutative the above diagram. 
\item[{\vdots}]
\end{itemize}
\end{assumption}

\begin{example} \label{ex:ass_Moore_CoMon}
For this example note that in any monoidal category $\mathsf M$ the
monoidal unit $I$ carries a trivial monoid structure which is initial in the
category of monoids in $\mathsf M$. 
Symmetrically, $I$ carries a trivial comonoid structure which is terminal in
the category of comonoids in $\mathsf M$. 
Whenever $\mathsf M$ is braided, the trivial monoid and comonoid structures
of $I$ combine to a bimonoid which is thus the zero object in the category of
bimonoids in $\mathsf M$. 

If moreover $\mathsf M$ has equalizers which are preserved by taking the
monoidal product with any object, then the category of bimonoids in $\mathsf M$
has equalizers --- see \cite[Example 3.3]{Bohm:Xmod_I} --- and thus
kernels. The kernel of any bimonoid morphism 
$\xymatrix@C=12pt{A \ar[r]^-f &B}$ is computed as the equalizer in $\mathsf M$ 
of  
\begin{equation} \label{eq:fhat}
\widehat f:=\xymatrix@C=12pt{
A \ar[r]^-\delta &
A^2 \ar[r]^-{\delta 1} &
A^3 \ar[r]^-{1f1} &
ABA}
\end{equation}
and 
$\widehat{u.\varepsilon}=
\xymatrix@C=12pt{A \ar[r]^-\delta & A^2 \ar[r]^-{1u1} & ABA}$,
see \cite{AndDev}.

So let $\mathsf M$ be a symmetric monoidal category in which equalizers
exist and are preserved by taking the monoidal product with
any object. Let $\mathsf C$ be the monoidal category of comonoids in $\mathsf
M$ and consider the monoidal admissible class $\mathcal S$ of spans in
$\mathsf C$ from \cite[Example 2.3]{Bohm:Xmod_I}.
Since $\varepsilon$ is the counit,
$$
\xymatrix@C=55pt{
A \ar[r]^-\delta \ar[d]_-\delta \ar[rrd]_-g &
A^2 \ar[r]^-c\ar[rd]^(.65){\varepsilon g} &
A^2 \ar[d]^-{g\varepsilon} \\
A^2 \ar[rr]_-{g\varepsilon} &&
B}
$$
commutes for any bimonoid morphism $g$ proving that
\begin{equation}\label{eq:eps_in_S}
\xymatrix@C=15pt{B & \ar[l]_-g A
\ar[r]^-\varepsilon & I}\in \mathcal S.
\end{equation}

From Example 3.3 and Proposition 3.7 in \cite{Bohm:Xmod_I} we know that all
$\mathcal S$-relative pullbacks exist in the category of bimonoids in $\mathsf
M$; hence
any simplicial bimonoid in $\mathsf M$
--- that is, any functor from $\Delta^{\mathsf{op}}$ to the category of
monoids in $\mathsf C$ --- satisfies the successive assumptions of Assumption
\ref{ass:Moore} for any positive integer. Still --- say, for an easier
comparison with \cite{Emir} --- below we present a more explicit description of
the objects $S_n^{(k)}$.

For any positive integer $n$ and any
$0< k \leq n$ the desired objects $S_n^{(k)}$ are constructed as the joint
kernels of the morphisms $\{\partial_n,\partial_{n-1},\dots , \partial_{n-k+1}\}$
in the category of bimonoids in $\mathsf M$; that is, as the joint equalizers
$$
\xymatrix@C=25pt{
S_n^{(k)} \ar[r]^-{j_n^{(k)}} &
S_n \ar@/^2.1pc/[rrr]^-{\widehat \partial_n}
\ar@/^1.4pc/[rrr]|-{\,\widehat\partial_{n-1}\,}_-\vdots
\ar@/_.7pc/[rrr]|-{\,\widehat\partial_{n-k+1}\,}
\ar@/_1.4pc/[rrr]_-{1u1.\delta} &&&
S_nS_{n-1}S_n}
$$
in $\mathsf M$ (where the hat notation of \eqref{eq:fhat} is used). By construction they are bimonoids.
Using the universality of the equalizer (in $\mathsf M$) in the bottom rows,
for $k>1$ we construct bimonoid morphisms in
$$
\xymatrix@C=20pt@R=45pt{
& S_n^{(k)} \ar[d]^-{j_n^{(k)}} \ar@{-->}[ld]_-{p_{S_n^{(k-1)}}} \\
S_n^{(k-1)} \ar[r]_-{j_n^{(k-1)}} &
S_n \ar@/^1.9pc/[rr]^-{\widehat\partial_n}
\ar@/^1.35pc/[rr]|(.55){\,\widehat\partial_{n-1}\,}_-\vdots
\ar@/_.55pc/[rr]|-{\,\widehat\partial_{n-k+2}\,}
\ar@/_1.2pc/[rr]_-{1u1.\delta} &&
S_nS_{n-1}S_n}
\xymatrix@C=20pt@R=45pt{
S_n^{(k-1)} \ar@{-->}[d]_-{\partial_{n-k+1}^{(k-1)}} \ar[r]^-{j_n^{(k-1)}} & 
S_n \ar[d]_-{\partial_{n-k+1}}  
\ar@/^1.9pc/[rr]^-{\widehat\partial_n}
\ar@/^1.3pc/[rr]|-{\,\widehat\partial_{n-1}\,}_-\vdots
\ar@/_.6pc/[rr]|-{\,\widehat\partial_{n-k+2}\,}
\ar@/_1.2pc/[rr]_(.4){1u1.\delta} &&
S_nS_{n-1}S_n \ar[d]|-{\ \ \partial_{n-k+1}\partial_{n-k+1}\partial_{n-k+1}} \\
S_{n-1}^{(k-1)} \ar[r]_-{j_{n-1}^{(k-1)}} &
S_{n-1} \ar@/^1.9pc/[rr]^(.4){\widehat\partial_{n-1}}
\ar@/^1.3pc/[rr]|-{\,\widehat\partial_{n-2}\,}_-\vdots
\ar@/_.6pc/[rr]|-{\,\widehat\partial_{n-k+1}\,}
\ar@/_1.2pc/[rr]_-{1u1.\delta} &&
S_{n-1}S_{n-2}S_{n-1}}
$$
(note the serial commutativity of the second diagram thanks to the simplicial
identities) and show that they give rise to the $\mathcal S$-relative
pullback
\begin{equation}\label{eq:S_n^k}
\xymatrix{
S_n^{(k)} \ar[r]^-\varepsilon \ar[d]_-{p_{S_n^{(k-1)}}} &
I \ar[d]^-u \\
S_n^{(k-1)} \ar[r]_-{\partial_{n-k+1}^{(k-1)}} &
S_{n-1}^{(k-1)}.}
\end{equation}
From \eqref{eq:eps_in_S} we infer 
$\xymatrix@C=15pt{S_{n-1}^{(k-1)}&& \ar[ll]_-{p_{S_n^{(k-1)}}}S_n^{(k)}
\ar[r]^-\varepsilon & I}\in \mathcal S$.
The square of \eqref{eq:S_n^k} commutes since $j_{n-1}^{(k-1)}$ in the right
vertical of the commutative diagram
$$
\xymatrix@R=5pt{
& S_n^{(k-1)} \ar[rr]^-{\partial_{n-k+1}^{(k-1)}} \ar[dd]^-{j_{n}^{(k-1)}} &&
S_{n-1}^{(k-1)} \ar[dd]^-{j_{n-1}^{(k-1)}} \\
\\
S_n^{(k)} \ar[r]^-{j_{n}^{(k)}} \ar[rd]^(.7){j_{n}^{(k)}}
\ar@/^1.2pc/[ruu]^-{p_{S_n^{(k-1)}}} \ar@/_1.2pc/[rrdd]_-\varepsilon &
S_n \ar[rr]^-{\partial_{n-k+1}} &&
S_{n-1} \\
& S_n \ar[rd]^-\varepsilon \\
&& I \ar[r]_-u \ar[ruu]^-u &
S_{n-1}^{(k-1)} \ar[uu]_-{j_{n-1}^{(k-1)}}}
$$
is a monomorphism.
In order to check the universality of \eqref{eq:S_n^k}, take a bimonoid
morphism $\xymatrix@C=12pt{C \ar[r]^-g & S_n^{(k-1)}}$ such that the exterior
of the first diagram of
\begin{equation}\label{eq:gtilde}
\xymatrix@C=15pt@R=15pt{
C \ar@{-->}[rd]^-{\widetilde g} \ar@/^1.3pc/[rrrd]^-\varepsilon 
\ar@/_1.3pc/[rddd]_(.6)g \\
& S_n^{(k)} \ar[rr]^-\varepsilon \ar[dd]_(.3){p_{S_n^{(k-1)}}} &&
I \ar[dd]^(.3)u \\
\\
& S_n^{(k-1)} \ar[rr]_-{\partial_{n-k+1}^{(k-1)}} &&
S_{n-1}^{(k-1)}}\quad
\raisebox{-5pt}{$
\xymatrix@R=15pt{
\qquad\qquad
\\
C \ar[r]^-g \ar@{-->}[dd]_-{\widetilde g} &
S_n^{(k-1)} \ar[dd]^(.4){j_n^{(k-1)}} \\
\\
S_n^{(k)}\ar[r]_-{j_n^{(k)}} &
S_n \ar@/^2.1pc/[rr]^-{\widehat\partial_n}
\ar@/^1.4pc/[rr]|-{\,\widehat\partial_{n-1}\,}_-\vdots
\ar@/_.6pc/[rr]|-{\,\widehat\partial_{n-k+1}\,}
\ar@/_1.4pc/[rr]_-{1u1.\delta=\widehat{u.\varepsilon}} &&
S_nS_{n-1}S_n }$}
\end{equation}
commutes (we know that
$\xymatrix@C=12pt{S_n^{(k-1)} & \ar[l]_-g C \ar[r]^-{\varepsilon} & I}\in
\mathcal S$ by \eqref{eq:eps_in_S}). Then a filler $\widetilde g$ of the first
diagram of \eqref{eq:gtilde} is constructed using the universality of the
equalizer in the bottom row of the second diagram of \eqref{eq:gtilde}. The
occurring morphism $j_n^{(k-1)}.g$ renders commutative the diagrams
$$
\xymatrix@R=43pt{
C \ar[r]^-g &
S_n^{(k-1)} \ar[r]^-{j_n^{(k-1)}} \ar[d]_-{j_n^{(k-1)}} &
S_n \ar[d]^-{\widehat \partial_i} \\
& S_n \ar[r]_-{1u1.\delta} &
S_nS_{n-1}S_n}\qquad
\xymatrix@C=20pt@R=10pt{
C \ar[r]^-g \ar[dd]_-g  \ar[rd]_-{u.\varepsilon} &
  S_n^{(k-1)} \ar[r]^-{j_n^{(k-1)}} \ar[d]^-{\partial_{n-k+1}^{(k-1)}} &
S_n \ar[dd]^-{\partial_{n-k+1}} \\
& S_{n-1}^{(k-1)} \ar[rd]^(.45){j_{n-1}^{(k-1)}} \\
S_n^{(k-1)} \ar[r]_-{j_n^{(k-1)}} &
S_n \ar[r]_-{u.\varepsilon} &
S_{n-1}}
$$
for $n-k+1<i\leq n$. 
Thus since it is a comonoid morphism, it equalizes the parallel morphisms of
the second diagram of \eqref{eq:gtilde}.  
The so constructed morphism $\widetilde g$ renders
commutative the first diagram of \eqref{eq:gtilde} since the right column and
the bottom row of the first commutative diagram in
\begin{equation}\label{eq:g_unique}
\xymatrix@C=35pt{
C \ar[r]^-{\widetilde g} \ar[d]_-g &
S_n^{(k)} \ar[r]^-{p_{S_n^{(k-1)}}} \ar[rd]_-{j_n^{(k)}} &
S_n^{(k-1)} \ar[d]^-{j_n^{(k-1)}} \\
S_n^{(k-1)} \ar[rr]_-{j_n^{(k-1)}} &&
S_n.}\quad
\xymatrix{
C \ar[d]_-h \ar[rr]^-{\widetilde g} \ar@/^1pc/[rd]^(.6)g &&
S_n^{(k)} \ar[d]^-{j_n^{(k)}} \\
S_n^{(k)} \ar[r]^-{p_{S_n^{(k-1)}}} \ar@/_1.1pc/[rr]_-{j_n^{(k)}}&
S_n^{(k-1)}\ar[r]^-{j_n^{(k-1)}} &
S_n}
\end{equation}
are equal monomorphisms. Finally $\tilde g$ is the unique filler of the first
diagram of \eqref{eq:gtilde}; as if also $h$ makes the first diagram of
\eqref{eq:gtilde} commute then also the second diagram of \eqref{eq:g_unique}
commutes. Since $j_n^{(k)}$ is a monomorphism, this proves $h=\widetilde g$.
In order to see the reflectivity of \eqref{eq:S_n^k} on the right, take
bimonoid morphisms 
$\xymatrix@C=12pt{D & \ar[l]_-h C \ar[r]^-g & S_n^{(k)}}$
such that
$\xymatrix@C=12pt{D & \ar[l]_-h C \ar[r]^-g & S_n^{(k)}
\ar[rr]^-{p_{S_n^{(k-1)}}} && S_n^{(k-1)}}\in \mathcal S$; equivalently, the
large square on the left of 
$$
\xymatrix@C=45pt@R=15pt{
C \ar[r]^-\delta \ar[dd]_-\delta &
C^2 \ar[r]^-c &
C^2 \ar[d]^-{hg} \\
&& DS_n^{(k)} \ar@/^1.2pc/[rd]^-{1j_n^{(k)}} \ar[d]^-{1p_{S_n^{(k-1)}}} \\
C^2 \ar[r]_-{hg} &
DS_n^{(k)} \ar[r]^-{1p_{S_n^{(k-1)}}} \ar@/_1.1pc/[rr]_-{1j_n^{(k)}} &
DS_n^{(k-1)} \ar[r]^-{1j_n^{(k-1)}} &
DS_n}
$$
commutes. Since $\xymatrix{DS_n^{(k-1)} \ar[r]^-{1j_n^{(k-1)}} & DS_n}$ is a
monomorphism, this is equivalent to the commutativity of the exterior. Since
also $\xymatrix{DS_n^{(k)} \ar[r]^-{1j_n^{(k)}} & DS_n}$ is a monomorphism,
this is further equivalent to $hg.c.\delta=hg.\delta$; that is, to 
$\xymatrix@C=12pt{D & \ar[l]_-h C \ar[r]^-g & S_n^{(k)}}\in \mathcal S$.
Reflectivity on the left follows symmetrically.
\end{example}

\begin{proposition} \label{prop:Moore}
Let $\mathcal S$ be a monoidal admissible class of spans in some 
monoidal category $\mathsf C$ and let $S$ be a simplicial monoid in $\mathsf
C$ which satisfies the successive conditions in Assumption \ref{ass:Moore} for
any positive integer.  
For any positive integer $n$, the morphisms
$D_{n-1}:=\xymatrix@C=30pt{S_{n}^{(n)} \ar[r]^-{p_{S_{n}^{(n-1)}}} &
S_{n}^{(n-1)} \ar[r]^-{\partial_0^{(n-1)}} & S_{n-1}^{(n-1)}}$ 
(where $S_n^{(0)}:=S_n$ and $\partial_0^{(0)}:=\partial_0$) render commutative
the diagram
$$
\xymatrix{
S_{n+1}^{(n+1)} \ar[r]^-{D_n} \ar[d]_-{p_I} &
S_n^{(n)} \ar[d]^-{D_{n-1}} \\
I \ar[r]_-u &
S_{n-1}^{(n-1)} .}
$$
In particular, whenever the monoidal unit $I$ is a zero object in the
category of monoids in $\mathsf C$, there is a chain complex
\begin{equation}\label{eq:Moore}
 \xymatrix{
 \cdots \ar[r]^-{D_{n+1}} &
 S_{n+1}^{(n+1)} \ar[r]^-{D_n} & 
 S_n^{(n)} \ar[r]^-{D_{n-1}} &
 \cdots \ar[r]^-{D_1} &
 S_1^{(1)} \ar[r]^-{D_0} &
 S_0.}
\end{equation}
\end{proposition}

\begin{proof}
The morphisms $D_n$ are clearly well defined and they render commutative
$$
\xymatrix@R=15pt@C=35pt{
S_{n+1}^{(n+1)} \ar[rr]_(.4){p_{S_{n+1}^{(n)}}} \ar[d]_-{p_I}
\ar@/^1.2pc/[rrrr]^-{D_n} &&
S_{n+1}^{(n)} \ar[rr]_(.66){\partial_0^{(n)}} 
\ar[ld]_(.7){\partial_1^{(n)}} \ar[rd]^(.7){p_{S_{n+1}^{(n-1)}}} && 
S_n^{(n)} \ar[d]_-{p_{S_n^{(n-1)}}} \ar@/^1.7pc/[ddd]^-{D_{n-1}} \\
I \ar[r]_-u \ar@{=}[dd] &
S_n^{(n)}\ar[rd]_-{p_{S_n^{(n-1)}}} &&
S_{n+1}^{(n-1)} \ar[r]_-{\partial_0^{(n-1)}} \ar[ld]^-{\partial_1^{(n-1)}} &
S_n^{(n-1)} \ar[dd]_-{\partial_0^{(n-1)}} \\
&& S_n^{(n-1)} \ar[d]^-{\partial_0^{(n-1)}} \\
I \ar[rr]_-u &&
S_{n-1}^{(n-1)} \ar@{=}[rr] &&
 S_{n-1}^{(n-1)}}
$$
\end{proof}

\begin{example}
The definition of Cartesian monoidal category includes the terminality of the monoidal unit; which is therefore a zero object in the category of monoids. Thus Proposition \ref{prop:Moore} implies the well-known fact that any simplicial
monoid in a Cartesian monoidal category admits a Moore complex \eqref{eq:Moore}.   
\end{example}

\begin{example} \label{ex:zero_CoMon}
If $\mathsf M$ is a symmetric monoidal category in which equalizers
exist and are preserved by taking the monoidal product with
any object, then we know from Example \ref{ex:ass_Moore_CoMon} that Assumption
\ref{ass:Moore} holds for any simplicial bimonoid in $\mathsf M$ (that is, for
any functor from $\Delta^{\mathsf{op}}$ to the category of monoids in the
category of comonoids in $\mathsf M$). 
Since in the category of bimonoids in $\mathsf M$ the monoidal unit is the zero
object, we conclude by Proposition \ref{prop:Moore} that any simplicial
bimonoid in $\mathsf M$ admits a Moore complex \eqref{eq:Moore}.   
\end{example}

\begin{definition} \label{def:Moore_length}
Let $\mathcal S$ be a monoidal admissible class of spans in some 
monoidal category $\mathsf C$. We say that a simplicial monoid $S$ in $\mathsf
C$ {\em has Moore length $l$} if the successive conditions in Assumption
\ref{ass:Moore} hold for any positive integer and 
$\xymatrix@C=12pt{I \ar[r]^-u & S_n^{(n)}}$ and
$\xymatrix@C=12pt{S_n^{(n)} \ar[r]^-{p_I} & I}$ are mutually inverse
isomorphisms for all $n>l$.
\end{definition}

\begin{lemma} \label{lem:Moore_length_iso}
Let $\mathcal S$ be a monoidal admissible class of spans in some 
monoidal category $\mathsf C$. For a simplicial monoid $S$ in $\mathsf C$ of
Moore length $l$, there are mutually inverse isomorphisms 
$\xymatrix@C=12pt{I \ar[r]^-u & S_n^{(n-i)}}$ and
$\xymatrix@C=12pt{S_n^{(n-i)} \ar[r]^-{p_I} &I}$ for any non-negative integer
$i$ and any $n>i+l$.
\end{lemma}

\begin{proof}
We proceed by induction in $i$.

For $i=0$ and $n>l$ the claim holds by definition.

Assume that it holds for some fixed value of $i$ and all $n>l+i$.

If $n>l+i+1$ then in the $\mathcal S$-relative pullback in the inner square of  
\begin{equation} \label{eq:p_inv} 
\xymatrix@C=15pt@R=15pt{
S_n^{(n-i-1)} \ar@{-->}[rd]^(.6){p_{S_n^{(n-i-1)}}^{-1}} \ar@/^1.4pc/[rrrd]^-{p_I}
\ar@{=}@/_1.5pc/[rddd] \\
& S_n^{(n-i)} \ar[rr]^-{p_I} \ar[dd]_(.3){p_{S_n^{(n-i-1)}}} &&
I \ar[dd]^-u \\
\\
& S_n^{(n-i-1)} \ar[rr]_-{\partial_{i+1}^{(n-i-1)}} &&
S_{n-1}^{(n-i-1)}}
\end{equation}
the right column is an isomorphism by the induction hypothesis. We claim that
so is then the left vertical of the inner square.
By the monoidality of $\mathcal S$, 
$\xymatrix@C=10pt{I & \ar@{=}[l] I \ar@{=}[r] & I}\in \mathcal S$, hence by
\cite[Lemma 3.4~(2)]{Bohm:Xmod_I} also
$\xymatrix@C=10pt{
S_n^{(n-i-1)} & \ar@{=}[l] S_n^{(n-i-1)}\ar[r]^-{p_I} & I}\in \mathcal S$. 
The exterior of \eqref{eq:p_inv} commutes by the commutativity of 
$$
\xymatrix{
S_n^{(n-i-1)} \ar[d]_-{\partial_{i+1}^{(n-i-1)}} \ar[r]^-{p_I} &
I \ar[d]^-u \\
S_{n-1}^{(n-i-1)} \ar[ru]^-{p_I} \ar@{=}[r] &
S_{n-1}^{(n-i-1)}}
$$
where the bottom triangle commutes by the induction hypothesis.
Then by the universality of the $\mathcal S$-relative pullback in
\eqref{eq:p_inv} there is a unique morphism $p_{S_n^{(n-i-1)}}^{-1}$ in 
\eqref{eq:p_inv}. It is the right inverse of $p_{S_n^{(n-i-1)}}$ by
construction and also the left inverse since 
$\xymatrix@C=15pt{S_n^{(n-i-1)} && \ar[ll]_-{p_{S_n^{(n-i-1)}}} S_n^{(n-i)}
\ar[r]^-{p_I} & I}$ are jointly monic and the following diagrams commute.
$$
\xymatrix@C=35pt{
&& S_n^{(n-i)} \ar[d]^-{p_{S_n^{(n-i-1)}}} \\
S_n^{(n-i)} \ar[r]_-{p_{S_n^{(n-i-1)}}} &
S_n^{(n-i-1)}\ar@/^1pc/[ru]^-{p_{S_n^{(n-i-1)}}^{-1}} \ar@{=}[r] &
S_n^{(n-i-1)}} \qquad
\xymatrix@C=35pt{
& S_n^{(n-i-1)}\ar[r]^-{p_{S_n^{(n-i-1)}}^{-1}} \ar[rd]_-{p_I} & 
S_n^{(n-i)} \ar[d]^-{p_I} \\
S_n^{(n-i)} \ar[rr]_-{p_I} \ar@/^1pc/[ru]^-{p_{S_n^{(n-i-1)}}} &&
I} 
$$
Using that also 
$\xymatrix@C=12pt{I \ar[r]^-u & S_n^{(n-i)}}$ is an isomorphism by the induction
hypothesis, so is the composite morphism 
$\xymatrix@C=10pt{
I \ar[r]^-u & 
S_n^{(n-i)} \ar[rrr]^-{p_{S_n^{(n-i-1)}}} &&&
S_n^{(n-i-1)}}
=
\xymatrix@C=12pt{I \ar[r]^-u &S_n^{(n-i-1)}}$
what completes the proof.
\end{proof}

\begin{lemma} \label{lem:k_up}
Let $\mathcal S$ be a monoidal admissible class of spans in some monoidal category $\mathsf C$ and let $S$ be a simplicial monoid in $\mathsf C$ for which the successive conditions of Assumption \ref{ass:Moore} hold for any positive integer. If there is some non-negative integer $k$ for which 
$\xymatrix@C=12pt{I \ar[r]^-u & S_n^{(k)}}$ and
$\xymatrix@C=12pt{S_n^{(k)} \ar[r]^-{p_I} &I}$ are mutually inverse isomorphisms  for all $n \geq k$, then also 
$\xymatrix@C=12pt{I \ar[r]^-u & S_n^{(k+1)}}$ and
$\xymatrix@C=12pt{S_n^{(k+1)} \ar[r]^-{p_I} &I}$ are mutually inverse isomorphisms  for all $n > k$.
\end{lemma}

\begin{proof}
We need to show that under the standing assumptions the inner square of
$$
\xymatrix@C=15pt@R=15pt{
X \ar@{-->}[rd]^-g  \ar@/^1.2pc/[rrrd]^-g \ar@/_1.2pc/[rddd]_-{f} \\
& I \ar@{=}[rr] \ar[dd]_-u &&
I \ar[dd]^-u  \\
\\
& S_n^{(k)} \ar[rr]_-{\partial_{n-k}^{(k)}} &&
 S_{n-1}^{(k)}}
$$
is an $\mathcal S$-relative pullback for all $n > k$. Commutativity of this square is immediate by the unitality of the morphism in the bottom row and
$\xymatrix@C=12pt{S_n^{(k)} & \ar[l]_-u I \ar@{=}[r] & I} \in \mathcal S$ by the unitality of $\mathcal S$. For the universality observe  that since the verticals are isomorphisms by assumption, the exterior of our diagram commutes if and only if $g$ is equal to 
$\xymatrix@C=16pt{X \ar[r]^-f & S_n^{(k)}  \ar[r]^-{\partial_{n-k}^{(k)}} &  S_{n-1}^{(k)} \ar[r]^-{p_I} & I}=
\xymatrix@C=16pt{X \ar[r]^-f & S_n^{(k)}  \ar[r]^-{p_I} & I}$; if and only if $f$ is equal to
$\xymatrix@C=16pt{X \ar[r]^-g & I\ar[r]^-u & S_n^{(k)}}$. Hence the (obviously
unique) filler is $g$. Since the top row of the inner square is the identity
morphism, reflectivity becomes trivial.
\end{proof}

\begin{corollary} \label{cor:Moore_length}
Let $\mathcal S$ be a monoidal admissible class of spans in some monoidal category $\mathsf C$. For a simplicial monoid $S$ in $\mathsf C$ which satisfies the successive conditions of Assumption \ref{ass:Moore} for any positive integer, the following assertions are equivalent.
\begin{itemize} 
\item[{(i)}] $S$ has Moore length $l$.
\item[{(ii)}] $\xymatrix@C=12pt{I \ar[r]^-u & S_n^{(l+1)}}$ and
$\xymatrix@C=12pt{S_n^{(l+1)} \ar[r]^-{p_I} &I}$ are mutually inverse isomorphisms  for all $n >l$.
\end{itemize}
\end{corollary}

\begin{proof}
Statement (i) implies (ii) by Lemma \ref{lem:Moore_length_iso} and (ii) implies (i) by Lemma \ref{lem:k_up}.
\end{proof}

\section{Invertibility of some canonical morphisms}
\label{sec:iso}

As in \cite[Theorems 1.1, 2.1 and 3.10]{Bohm:Xmod_II}, also in the forthcoming equivalence between relative categories and certain simplicial monoids a crucial role is played by the invertibility of some canonical morphisms discussed in this section. 

Consider a monoidal admissible class of spans in an arbitrary monoidal category $\mathsf C$.
Take a simplicial monoid $S$ as in \eqref{eq:simp_obj} in $\mathsf C$ for which the successive conditions in Assumption \ref{ass:Moore} hold for any positive integer. For any non-negative integer $n$, for $0\leq i \leq n$ and for $0\leq k \leq n-i$, we define a morphism 
$\xymatrix@C=15pt{S_n^{(k)} \ar[r]^-{\sigma_i^{(k)}} & S_{n+1}^{(k)}}$ iteratively as follows.
\begin{itemize}
\item $\sigma_i^{(0)}:=\sigma_i$.
\item For $0<k \leq n-i$ we define $\sigma_i^{(k)}$ as the unique filler in
$$
\xymatrix@R=10pt@C=25pt{
S_n^{(k)} \ar@/^1.2pc/[rrrd]^-{p_I} \ar[ddd]_-{p_{S_n^{(k-1)} }} \ar@{-->}[rd]^(.55){\sigma_i^{(k)}} \\
& S_{n+1}^{(k)} \ar[rr]^-{p_I} \ar[dd]_-{p_{S_{n+1}^{(k-1)} }} &&
I \ar[dd]^-u \\
\\
S_n^{(k-1)}  \ar[r]_-{\sigma_i^{(k-1)}} &
S_{n+1}^{(k-1)} \ar[rr]_-{\partial_{(n+1)-(k-1)}^{(k-1)}} &&
S_{n}^{(k-1)}}
$$
It is well-defined by the simplicial identity $\partial_{(n+1)-(k-1)}.\sigma_i=\sigma_i.\partial_{n-(k-1)}$ and the unitality of $\sigma_i^{(k-1)}$, see \cite[Proposition 3.5]{Bohm:Xmod_I}.
\end{itemize}
With these morphisms $\sigma_i^{(k)}$ at hand, we introduce for any positive integer $n$ and $0\leq k < n$ the morphisms
\begin{equation} \label{eq:y_nk}
y_{(n,k)}:=
\xymatrix@C=20pt{
S_n^{(k+1)} S_{n-1}^{(k)} \ar[rrr]^-{p_{S_n^{(k)}} \sigma_{n-1-k}^{(k)}} &&&
(S_n^{(k)})^2 \ar[r]^-m &
S_n^{(k)}.}
\end{equation}
Note that they are natural in the following sense. For a simplicial monoid
morphism $\{\xymatrix@C=16pt{S_n \ar[r]^-{f_n} & S'_n}\}_{n\geq 0}$, let us
define inductively the morphisms $f_n^{(0)}:=f_n$ and for $0<k \leq n$ the unique
morphism $f_n^{(k)}$ which renders commutative
$$
\xymatrix@C=15pt@R=15pt{
S_n^{(k)} \ar@/^1.2pc/[rrrd]^-{p_I} \ar[ddd]_-{p_{S_n^{(k-1)}}} 
\ar@{-->}[rd]^-{f_n^{(k)}}\\
& S_n^{\prime (k)} \ar[rr]^-{p_I} \ar[dd]_-{p_{S_n^{\prime (k-1)}}} &&
I \ar[dd]^-u \\
\\
S_n^{(k-1)} \ar[r]_-{f_n^{(k-1)}} &
S_n^{\prime (k-1)} \ar[rr]_-{\partial_{n-k+1}^{\prime (k-1)}} &&
S_{n-1}^{\prime (k-1)}.}
$$
(This definition makes sense --- see \cite[Proposition 3.5]{Bohm:Xmod_I} ---
since $\{f_n\}_{n\geq 0}$ is a morphism of simplicial monoids by
assumption, hence so is $\{f_n^{(k-1)}\}_{n\geq 0}$ and therefore
$$
\xymatrix@C=35pt@R=20pt{
S_n^{(k-1)} \ar[r]^-{\partial_{n-k+1}^{(k-1)}} \ar[d]_-{f_n^{(k-1)}} &
S_{n-1}^{(k-1)} \ar[d]_-{f_{n-1}^{(k-1)}} &
I \ar[l]_-u \ar@{=}[d] \\
S_n^{\prime (k-1)} \ar[r]_-{\partial_{n-k+1}^{\prime (k-1)}}  &
S_{n-1}^{\prime (k-1)}  &
I \ar[l]^-u }
$$
commutes.)
These morphisms and $y_{(n,k)}$ fit into the commutative the diagram 
\begin{equation}\label{eq:y_nat}
\xymatrix@C=35pt{
S_n^{(k+1)}S_{n-1}^{(k)} \ar[rr]_-{p_{S_n^{(k)}} \sigma_{n-k-1}^{(k)}}
\ar[d]_-{f_n^{(k+1)} f_{n-1}^{(k)}} \ar@/^1.1pc/[rrr]^-{y_{(n,k)}} &&
(S_n^{(k)})^2\ar[r]_-m \ar[d]^-{f_n^{(k)}f_n^{(k)}} &
S_n^{(k)} \ar[d]^-{f_n^{(k)}} \\
S_n^{\prime (k+1)}S_{n-1}^{\prime (k)} 
\ar[rr]^-{p_{S_n^{\prime (k)}} \sigma_{n-k-1}^{\prime (k)}}
\ar@/_1.1pc/[rrr]_-{y'_{(n,k)}} &&
(S_n^{\prime (k)})^2\ar[r]^-m &
S_n^{\prime (k)} .}
\end{equation}

\begin{lemma}
Consider a monoidal admissible class of spans in an arbitrary monoidal category $\mathsf C$.
Take a simplicial monoid $S$ as in \eqref{eq:simp_obj} in $\mathsf C$ for which the successive conditions in Assumption \ref{ass:Moore} hold for any positive integer.  
If for some positive integer $n$ and some $0\leq k < n-1$ the morphism $y_{(n,k)}$ of \eqref{eq:y_nk} is invertible then also $y_{(n-1,k)}$ is invertible. 
\end{lemma}

\begin{proof}
It follows by the commutativity of both diagrams
\begin{equation}\label{eq:y_id}
\xymatrix@C=10pt{
S_n^{(k+1)} S_{n-1}^{(k)}  \ar[rrr]_-{p_{S_n^{(k)}} \sigma_{n-1-k}^{(k)}}  \ar@/^1.1pc/[rrrr]^-{y_{(n,k)}} \ar[d]_-{\partial_i^{(k+1)} \partial_i^{(k)}} &&&
(S_n^{(k)})^2 \ar[r]_-m \ar[d]^-{\partial_i^{(k)} \partial_i^{(k)}} &
S_n^{(k)} \ar[d]^-{\partial_i^{(k)}} \\
S_{n-1}^{(k+1)} S_{n-2}^{(k)} \ar[rrr]^-{p_{S_{n-1}^{(k)}} \sigma_{n-2-k}^{(k)}} \ar@/_1.1pc/[rrrr]_-{y_{(n-1,k)}} &&&
(S_{n-1}^{(k)})^2 \ar[r]^-m &
S_{n-1}^{(k)}}
\xymatrix@C=10pt{
S_{n-1}^{(k+1)} S_{n-2}^{(k)} \ar[rrr]_-{p_{S_{n-1}^{(k)}} \sigma_{n-2-k}^{(k)}} \ar@/^1.1pc/[rrrr]^-{y_{(n-1,k)}} \ar[d]_-{\sigma_i^{(k+1)} \sigma_i^{(k)}} &&&
(S_{n-1}^{(k)})^2 \ar[r]_-m \ar[d]^-{\sigma_i^{(k)} \sigma_i^{(k)}} &
S_{n-1}^{(k)} \ar[d]^-{\sigma_i^{(k)}} \\
S_n^{(k+1)} S_{n-1}^{(k)}  \ar[rrr]^-{p_{S_n^{(k)}} \sigma_{n-1-k}^{(k)}}  \ar@/_1.1pc/[rrrr]_-{y_{(n,k)}} &&&
(S_n^{(k)})^2 \ar[r]^-m &
S_n^{(k)} }
\end{equation}
for any $0\leq i < n-k-1$  together with the simplicial identity $\partial_i.\sigma_i=1$ that $y_{(n-1,k)}$ has the inverse
$$
\xymatrix@C=35pt{
S_{n-1}^{(k)} \ar[r]^-{\sigma_i^{(k)}} &
S_n^{(k)} \ar[r]^-{y_{(n,k)}^{-1}} &
S_n^{(k+1)} S_{n-1}^{(k)}  \ar[r]^-{\partial_i^{(k+1)} \partial_i^{(k)}} &
S_{n-1}^{(k+1)} S_{n-2}^{(k)}.}
$$
\end{proof}

Note that in addition to the diagrams of \eqref{eq:y_id}, $y_{(n,k)}$ renders commutative also
\begin{equation} \label{eq:y_more}
\xymatrix@C=35pt{
S_n^{(k+1)} S_{n-1}^{(k)} \ar[r]_-{p_{S_n^{(k)}}1} \ar@/^1.1pc/[rrr]^-{y_{(n,k)}} 
\ar[d]_-{\partial_{n-k-1}^{(k+1)} 1} &
S_n^{(k)} S_{n-1}^{(k)} \ar[r]_(.65){1 \sigma_{n-1-k}^{(k)}}   
\ar[rd]_(.25){\partial_{n-k-1}^{(k)} 1} &
(S_n^{(k)})^2 \ar[r]_-m \ar[d]^-{\partial_{n-k-1}^{(k)} \partial_{n-k-1}^{(k)}} &
S_n^{(k)} \ar[d]^-{\partial_{n-k-1}^{(k)}} \\
S_{n-1}^{(k+1)} S_{n-1}^{(k)} \ar[rr]_-{p_{S_{n-1}^{(k)}}1} &&
(S_{n-1}^{(k)})^2 \ar[r]_-m &
S_{n-1}^{(k)} &
\textrm{and} \\
S_n^{(k+1)} S_{n-1}^{(k)}  \ar[r]_-{p_{S_n^{(k)}}1}  \ar@/^1.1pc/[rrr]^-{y_{(n,k)}} \ar[d]_-{p_I 1} &
S_n^{(k)} S_{n-1}^{(k)} \ar[r]_-{1 \sigma_{n-1-k}^{(k)}}   \ar[d]^-{\partial_{n-k}^{(k)} 1} &
(S_n^{(k)})^2 \ar[r]_-m \ar[d]^-{\partial_{n-k}^{(k)} \partial_{n-k}^{(k)}} &
S_n^{(k)} \ar[d]^-{\partial_{n-k}^{(k)}} \\
S_{n-1}^{(k)} \ar[r]^-{u1} \ar@{=}@/_1.1pc/[rrr] &
(S_{n-1}^{(k)})^2 \ar@{=}[r] &
(S_{n-1}^{(k)})^2 \ar[r]^-m &
S_{n-1}^{(k)}.}
\end{equation}
\smallskip

\begin{proposition} \label{prop:l1_y}
Let $\mathcal S$ be a monoidal admissible class of spans in some 
monoidal category $\mathsf C$. For a simplicial monoid $S$ in $\mathsf C$ of
Moore length $l$, the following assertions hold.
\begin{itemize}
\item[{(1)}] The morphism $y_{(n,k)}$ of \eqref{eq:y_nk} is invertible for all $l<k<n$.
\item[{(2)}] For any $n>l$ the following are equivalent.
\begin{itemize}
\item[{(i)}]  $y_{(n,l)}$ of \eqref{eq:y_nk} is invertible.
\item[{(ii)}] 
$\xymatrix@C=35pt{S_{n-1}^{(l)} \ar[r]^-{\sigma_{n-1-l}^{(l)}} & S_n^{(l)}}$ is invertible.
\item[{(iii)}] The morphisms 
$\xymatrix@C=70pt{
S_{n-1}^{(l)} \ar[r]^-{\sigma_0^{(l)},\sigma_1^{(l)},\dots,\sigma_{n-1-l}^{(l)}} &
S_n^{(l)}}$ are equal isomorphisms with the inverse
$\partial_0^{(l)}=\partial_1^{(l)}=\dots=\partial_{n-l}^{(l)}$.
\end{itemize}
\end{itemize}
\end{proposition}

\begin{proof}
(1) In the unitality diagram
$$
\xymatrix@C=35pt{
I \ar@{=}[r] \ar[d]_-{uu} &
I \ar[d]^-u \\
S_n^{(k+1)} S_{n-1}^{(k)} \ar[r]_-{y_{(n,k)}} &
S_n^{(k)}}
$$
the verticals are isomorphisms by Lemma \ref{lem:Moore_length_iso}. Then the bottom row is an isomorphism too.

(2) The equivalence of (i) and (ii) follows by the commutativity of the diagram 
$$
\xymatrix@C=15pt{
S_{n-1}^{(l)} \ar[rrr]^-{\sigma_{n-1-l}^{(l)}} \ar[d]_-{u1} &&&
S_n^{(l)} \ar[d]^-{u1} \ar@/^1.2pc/@{=}[rd] \\
S_n^{(l+1)} S_{n-1}^{(l)} \ar@/_1.1pc/[rrrr]_-{y_{(n,l)}} \ar[rrr]^-{p_{S_n^{(l)}} \sigma_{n-1-l}^{(l)}} &&&
(S_n^{(l)})^2 \ar[r]^-m &
S_n^{(l)}}
$$
whose left vertical is an  isomorphism by Lemma \ref{lem:Moore_length_iso}.

Assertion (iii) trivially implies (ii). Conversely, if $\sigma_{n-1-l}^{(l)}$ is invertible then by the simplicial relations its inverse is $\partial_{n-1-l}^{(l)}=\partial_{n-l}^{(l)}$ which is then invertible too. Again by the simplicial relations the inverse of $\partial_{n-1-l}^{(l)}=\partial_{n-l}^{(l)}$ is
$\sigma_{n-1-l}^{(l)}=\sigma_{n-2-l}^{(l)}$. Iterating this reasoning we conclude that (iii) holds.
\end{proof}

\begin{example} \label{ex:y_groupoid}
For a simplicial category $S$ (that is, a functor $S$ from
$\Delta^{\mathsf{op}}$ to the category of monoids in the category of spans
over a fixed set in \cite[Example 1.2]{Bohm:Xmod_II}), $S_n^{(k+1)}$ is
the subcategory of those morphisms in $S_n^{(k)}$ which are taken by the
functor $\partial_{n-k}^{(k)}$ to an identity morphism. Hence for a simplicial
{\em groupoid} $S$, the category $S_n^{(k)}$ is a groupoid for all $0\leq k
\leq n$. In this case all morphisms $\{y_{(n,k)}\}_{0\leq k<n}$ of
\eqref{eq:y_nk} are invertible by the same argument applied to the morphism
(1.3) in \cite[Example 1.2]{Bohm:Xmod_II}. 
\end{example}

\begin{example} \label{ex:y_Hopf}
Let $\mathsf M$ be a symmetric monoidal category in which equalizers 
exist and are preserved by taking the monoidal product with any object. 
Let $\mathsf C$ be the monoidal category of comonoids in $\mathsf M$ and let
$\mathcal S$ be the monoidal admissible class of spans in $\mathsf C$ from
\cite[Example 2.3]{Bohm:Xmod_I}. 
Take a simplicial monoid $S$ in $\mathsf C$ (that is, a simplicial bimonoid in
$\mathsf M$) such that for all non-negative integers $n$, $S_n$ is a
cocommutative Hopf monoid in $\mathsf M$. Then $y_{(n,k)}$ of \eqref{eq:y_nk}
is invertible for all $0\leq k <n$.  This can be seen as follows.  

Recall from Example \ref{ex:ass_Moore_CoMon} that $S_n^{(k)}$ is now a joint
kernel in the category of bimonoids in $\mathsf M$, hence it is a sub bimonoid
of the cocommutative bimonoid $S_n$. Thus $S_n^{(k)}$ is a cocommutative
bimonoid. Moreover, by the cocommutativity of $S_n$, its antipode is a
comonoid morphism 
$\xymatrix@C=12pt{S_n\ar[r]^-{z_n} & S_n}$ for all $n\geq 0$.  
So we can use the universality of the equalizer in $\mathsf C$ in the bottom
row of
\begin{equation} \label{eq:z_nk}
\xymatrix@C=40pt@R=40pt{
S_n^{(k)} \ar@{-->}[d]_-{z_n^{(k)}} \ar[r]^-{j_n^{(k)}} & 
S_n \ar[d]_-{z_n}  
\ar@/^1.9pc/[rr]^-{\partial_n}
\ar@/^1.3pc/[rr]|-{\,\partial_{n-1}\,}_-\vdots
\ar@/_.6pc/[rr]|-{\,\partial_{n-k+1}\,}
\ar@/_1.2pc/[rr]_(.4){u.\varepsilon} &&
S_{n-1} \ar[d]^-{z_{n-1}} \\
S_n^{(k)} \ar[r]_-{j_n^{(k)}} &
S_n \ar@/^1.9pc/[rr]^(.4){\partial_n}
\ar@/^1.3pc/[rr]|-{\,\partial_{n-1}\,}_-\vdots
\ar@/_.6pc/[rr]|-{\,\partial_{n-k+1}\,}
\ar@/_1.2pc/[rr]_-{u.\varepsilon} &&
S_{n-1}}
\end{equation}
to define the antipode $z_n^{(k)}$ for all $n>0$ and $0\leq k\leq n$ as the
restriction of $z_n$. The diagram of \eqref{eq:z_nk} is serially commutative
since any bimonoid morphism between Hopf monoids is compatible with the
antipodes, thus so are in particular the parallel morphisms of the rows.
This makes $S_n^{(k)}$ a cocommutative Hopf monoid for all $0\leq k \leq n$.
By construction the morphisms 
$\xymatrix{S_n^{(k)} \ar[r]^-{\partial_{n-k}^{(k)}} & S_{n-1}^{(k)}}$ and 
$\xymatrix{S_{n-1}^{(k)} \ar[r]^-{\sigma_{n-k-1}^{(k)}} & S_{n^{(k)}}}$
are morphisms of bimonoids and therefore of Hopf monoids.

Summarizing,
$\xymatrix@C=35pt{
S_{n-1}^{(k)}  \ar@{ >->}@<2pt>^-{\, \sigma_{n-k-1}^{(k)} \,}[r] &
S_{n}^{(k)} \ar@{->>}@<2pt>[l]^-{\partial_{n-k}^{(k)}} }$
is a split epimorphism of bimonoids in $\mathsf M$ for which conditions (1.a)
and (1.b) of \cite[Proposition 1.4]{Bohm:Xmod_II} hold; hence by 
\cite[Proposition 1.4]{Bohm:Xmod_II} (see also \cite{Radford}) the corresponding morphism $y_{(n,k)}$ in \cite[Theorem 1.1~(1.b)]{Bohm:Xmod_II} is invertible.

The above assumption about the cocommutativity of each $S_n$ may look quite strong. Note however, that for the application of \cite[Proposition 1.4]{Bohm:Xmod_II} wee need the assumption that 
$\xymatrix@C=20pt{S_n^{(k)}  & \ar@{=}[l] S_n^{(k)}  \ar[r]^-{\partial_{n-k}^{(k)}} & S_{n-1}^{(k)}}$ belongs to the class $\mathcal S$. By \cite[Lemma 2.4]{Bohm:Xmod_I} this is equivalent to the cocommutativity of $S_{n-1}^{(k)}$. This should hold for all $n>0$ and $0\leq k <n$; so in particular for $k=0$.
\end{example}

\section{Equivalence of relative categories and simplicial monoids of Moore length 1}
\label{sec:equivalence}

\begin{theorem} \label{thm:SCat_vs_Simp}
Consider a monoidal admissible class $\mathcal S$ of spans in a monoidal
category $\mathsf C$ such that 
there exist the $\mathcal S$-relative pullbacks of those cospans in $\mathsf
C$ whose legs are in $\mathcal S$ (cf. \cite[Assumption 4.1]{Bohm:Xmod_I}). 
The equivalent categories $\mathsf{CatMon}_{\mathcal S}(\mathsf C)$ and
$\mathsf{Xmod}_{\mathcal S}(\mathsf C)$ of \cite[Theorem 3.10]{Bohm:Xmod_II}
are equivalent also to the category  
\begin{itemize}
\item[{$\mathsf{Simp}$}]\hspace{-.3cm} ${}^1\mathsf{Mon}_{\mathcal S}(\mathsf C)$ whose \\
\underline{objects} are simplicial monoids in $\mathsf C$  (that, is functors $S$ from $\Delta^{\mathsf{op}}$ to the category of monoids in $\mathsf C$) such that the following conditions hold. 
\begin{itemize}
\item[{(a)}] $S$ has Moore length 1.
\item[{(b)}] Using the notation from \eqref{eq:simp_obj}, 
$\xymatrix@C=12pt{S_0 & \ar[l]_-{\partial_1} S_1 \ar@{=}[r] & S_1}$ and 
$\xymatrix@C=12pt{S_1 & \ar@{=}[l] S_1 \ar[r]^-{\partial_0} & S_0}$ belong to $\mathcal S$.
\item[{(c)}] The morphisms $y_{(n,k)}$ of \eqref{eq:y_nk} and the morphisms 
$$
q_n:=\xymatrix@C=20pt{
S_1^{(1)}S_1^{\expcoten {S_0} n-1} \ar[r]^-{p_{S_1}1} &
S_1 S_1^{\expcoten {S_0} n-1} \ar[rrrr]^-{(1\diagcoten \sigma_0 \diagcoten \cdots \diagcoten \sigma_0)(\sigma_0 \diagcoten 1)} &&&&
(S_1^{\expcoten {S_0} n})^2 \ar[r]^-m &
S_1^{\expcoten {S_0} n}}
$$
are invertible for all $n> 0$ and $0\leq k < n$.
\end{itemize}
\underline{morphisms} are simplicial monoid morphisms (that is, natural transformations between the  functors from $\Delta^{\mathsf{op}}$ to the category of monoids in $\mathsf C$).
\end{itemize}
\end{theorem}

\begin{proof}
The proof consists of the construction of mutually inverse equivalence
functors between $\mathsf{Simp}^1\mathsf{Mon}_{\mathcal S}(\mathsf C)$ and
$\mathsf{CatMon}_{\mathcal S}(\mathsf C)$. 

The functor $\mathsf{Simp}^1\mathsf{Mon}_{\mathcal S}(\mathsf C) \to \mathsf{CatMon}_{\mathcal S}(\mathsf C)$ sends an object $S$ in \eqref{eq:simp_obj} to the $\mathcal S$-relative category whose underlying reflexive graph is 
$
\xymatrix@C=30pt{
S_0 \ar[r]|-{\,\sigma_0\,} &
\ar@<-5pt>[l]_-{\partial_0}\ar@<5pt>[l]^-{\partial_1} S_1.}
$
By construction this is an object of the category $\mathsf{ReflGraphMon}_{\mathcal S}(\mathsf C)$ in \cite[Theorem 2.1]{Bohm:Xmod_II}  for which the morphisms $q_n$ of \cite[(3.3)]{Bohm:Xmod_II} are invertible. 
By \cite[Proposition 3.8]{Bohm:Xmod_II} it extends uniquely to an object of $\mathsf{CatMon}_{\mathcal S}(\mathsf C)$ since the following diagram commutes. 
$$
\xymatrix@C=20pt@R=15pt{
S_1 S_1^{(1)} \ar[d]_-{1p_{S_1}} \\
S_1^2 \ar[r]^-{\raisebox{8pt}{${}_{(\sigma_0 \diagcoten 1)(1\diagcoten \sigma_0)}$}} \ar[d]^-{1y_{(1,0)}^{-1}=1q_1^{-1}} 
\ar@{=}@/_2.5pc/[dddd] &
(S_1\coten {S_0} S_1)^2 \ar[rrr]^-{\raisebox{8pt}{${}_m$}} \ar[d]^-{q_2^{-1} q_2^{-1}} &&&
S_1\coten {S_0} S_1 \ar[rr]^-{\raisebox{8pt}{${}_{q_2^{-1}}$}} &&
S_1^{(1)} S_1 \ar[ld]_(.7){\sigma_0^{(1)} 1} \ar[dddd]^-{p_{S_1} 1} \\
S_1 S_1^{(1)} S_0 \ar[r]^-{u11\sigma_0} \ar[d]^-{1p_{S_1}\sigma_0} &
(S_1^{(1)}S_1)^2 \ar[rr]^-{\sigma_0^{(1)} 1 \sigma_0^{(1)} 1} \ar[d]^-{p_{S_1} 1 p_{S_1} 1}  &&
(S_2^{(1)}S_1)^2 \ar@{-->}[rr]^-m \ar[d]^-{y_{(2,0)}y_{(2,0)}} &&
S_2^{(1)}S_1 \ar[ld]_(.7){y_{(2,0)}} \\
S_1^3 \ar[r]^-{u111}\ar[dd]^-{1m} &
S_1^4 \ar[r]^-{\sigma_0 \sigma_1 \sigma_0 \sigma_1} \ar@{=}[rd] &
S_2^4 \ar[r]^-{mm} \ar[d]^-{\partial_1 \partial_1 \partial_1 \partial_1} &
S_2^2 \ar[r]^-m &
S_2 \ar[dd]^-{\partial_1} \\
&& S_1^4 \ar[d]^-{mm} &&&
S_2^2 \ar[ul]_-m \ar[d]_-{\partial_1 \partial_1} \\
S_1^2 \ar@{=}[rr] &&
S_1^2 \ar[rr]_-m &&
S_1 &
S_1^2 \ar[l]^-m \ar@{=}[r] &
S_1^2 \ar[lu]_-{\sigma_0 \sigma_1} }
$$
Recall from Proposition \ref{prop:l1_y} that 
$\xymatrix@C=15pt{S_1^{(1)} \ar[r]^-{\sigma_0^{(1)}} & S_2^{(1)}}$ and
$\xymatrix@C=15pt{S_2^{(1)} \ar[r]^-{\partial_0^{(1)}} & S_1^{(1)}}$ are mutually inverse isomorphisms.
Hence the regions of  
the above diagram sharing the dashed arrow commute because both $q_2.\partial_0^{(1)}1$ and $y_{(2,0)}$ are multiplicative by \cite[Lemma 1.5]{Bohm:Xmod_I}, with respect to the multiplications induced by the respective distributive laws
$$
\xymatrix@C=15pt @R=10pt{
S_1 S_2^{(1)} \ar[r]^-{1\partial_0^{(1)}} &
S_1 S_1^{(1)} \ar[r]^-{1p_{S_1}} &
S_1^2 \ar[rrr]^-{(\sigma_0\diagcoten 1)(1\diagcoten \sigma_0)} &&&
(S_1\coten {S_0} S_1)^2 \ar[r]^-m &
S_1\coten {S_0} S_1 \ar[r]^-{q_2^{-1}} &
S_1^{(1)} S_1 \ar[r]^-{\sigma_0^{(1)} 1} &
S_2^{(1)} S_1\\
& \textrm{and}&
S_1 S_2^{(1)} \ar[rr]^-{\sigma_1 p_{S_2}} &&
S_2^2 \ar[r]^-m &
S_2 \ar[r]^-{y_{(2,0)}^{-1}} &
S_2^{(1)}S_1 }
$$
whose equality follows by the commutativity of the diagrams 
$$
\xymatrix@R=12pt@C=15pt{
&& S_2S_2^{(1)} \ar[r]^-{1p_{S_2}} \ar[d]^-{\partial_0\partial_0^{(1)}} &
S_2^2 \ar[r]^-m \ar[dd]^-{\partial_0 \partial_0} &
S_2 \ar[r]^-{y_{(2,0)}^{-1}} \ar[dd]^-{\partial_0} \ar@{}[rrdd]|-{\eqref{eq:y_id}} &
S_2^{(1)}S_1 \ar[r]^-{\partial_0^{(1)}1} &
S_1^{(1)} S_1 \ar[r]^-{q_2} \ar[dd]_-{1\partial_0} 
\ar@{}[rdd]|-{\txt{\tiny \cite[(3.8)]{Bohm:Xmod_II}}
} &
S_1\coten {S_0} S_1 \ar[dd]^-{p_1} \\
&& S_1 S_1^{(1)} \ar@/^.8pc/[rd]^-{1p_{S_1}} \\
S_1 S_1^{(1)} \ar[r]^-{\partial_0 1} \ar@/^1.2pc/[rruu]^-{\sigma_1 \sigma_0^{(1)}} \ar@/_1.2pc/[rrdd]_-{1p_{S_1}} &
S_0 S_1^{(1)} \ar@/^.8pc/[ru]^-{\sigma_0 1} \ar@/_.8pc/[rd]_-{1p_{S_1}} &&
S_1^2 \ar[r]^-m &
S_1 \ar[rr]^-{y_{(1,0)}^{-1}=q_1^{-1}} \ar@{=}@/_1.1pc/[rrr] &&
S_1^{(1)} S_0 \ar[r]^-{q_1} &
S_1 \\
&& S_0S_1 \ar@/_.8pc/[ru]_-{\sigma_0 1} \\
&& S_1^2 \ar[u]_-{\partial_0 1} \ar[r]_-{\raisebox{-8pt}{${}_{(\sigma_0\diagcoten 1)(1\diagcoten \sigma_0)}$}} &
(S_1\coten {S_0} S_1)^2 \ar[uu]_-{p_1p_1} \ar[rrrr]_-{\raisebox{-8pt}{${}_m$}} &&&&
S_1\coten {S_0} S_1 \ar[uu]_-{p_1}}
$$
$$
\xymatrix@R=25pt@C=15pt{
& S_2S_2^{(1)} \ar[rr]^-{1p_{S_2}} \ar[d]^-{\partial_2p_I} &&
S_2^2 \ar[r]^-m \ar[d]^-{\partial_2 \partial_2} &
S_2 \ar[r]^-{y_{(2,0)}^{-1}} \ar[d]_-{\partial_2} \ar@{}[rd]|-{\eqref{eq:y_more}} &
S_2^{(1)}S_1 \ar[r]^-{\partial_0^{(1)}1} \ar[d]^-{p_I 1} &
S_1^{(1)} S_1 \ar[r]^-{q_2} \ar[d]_-{p_I 1} \ar@{}[rd]|-{
\txt{\tiny \cite[(3.9)]{Bohm:Xmod_II}
}} &
S_1\coten {S_0} S_1 \ar[d]^-{p_2} \\
S_1 S_1^{(1)} \ar[r]^-{1p_I} \ar@/^1.2pc/[ru]^-{\sigma_1 \sigma_0^{(1)}} 
\ar@/_1.2pc/[rrd]_-{1p_{S_1}} &
S_1 \ar@/^.8pc/[rr]^-{1u} \ar[r]_-{1u} &
S_1S_0 \ar[r]_-{1\sigma_0} &
S_1^2 \ar[r]^-m &
S_1 \ar@{=}[r] &
S_1 \ar@{=}[r] &
S_1 \ar@{=}[r] &
S_1 \\
&& S_1^2 \ar[u]_-{1 \partial_1} 
\ar[r]_-{\raisebox{-8pt}{${}_{(\sigma_0\diagcoten 1)(1\diagcoten \sigma_0)}$}} &
(S_1\coten {S_0} S_1)^2 \ar[u]_-{p_2p_2} 
\ar[rrrr]_-{\raisebox{-8pt}{${}_m$}} &&&&
S_1\coten {S_0} S_1. \ar[u]_-{p_2}}
$$

A morphism 
$\xymatrix@C=12pt{S \ar[r]^-F &  S'}$ in $\mathsf{Simp}^1\mathsf{Mon}_{\mathcal S}(\mathsf C)$ is sent to the morphism of reflexive graphs
$(\xymatrix@C=12pt{S_0 \ar[r]^-{F_0} & S'_0},\xymatrix@C=12pt{S_1 \ar[r]^-{F_0} & S'_1})$; it is an $\mathcal S$-relative functor by \cite[Proposition 3.9]{Bohm:Xmod_II}.

In the opposite direction $\mathsf{CatMon}_{\mathcal S}(\mathsf C) \to \mathsf{Simp}^1\mathsf{Mon}_{\mathcal S}(\mathsf C)$, we send an $\mathcal S$-relative category 
$\xymatrix@C=20pt{
B \ar@{ >->}|(.55){\, i \,}[r] &
A \ar@{->>}@<-5pt>[l]_-s \ar@{->>}@<5pt>[l]^- t &
A \coten B A \ar[l]_-d}$
to its `$\mathcal S$-relative nerve'
$$
\xymatrix@C=30pt{
B \ar[r]|-{\,i \,} &
\ar@<-5pt>[l]_-{t}\ar@<5pt>[l]^-{s} A
\ar@<7pt>[r]|-{\,1\diagcoten i \,} \ar@<-7pt>[r]|-{\,i\diagcoten 1 \,} &
\ar[l]|-{\, d \,}
\ar@<-12pt>[l]_-{\, p_1 \,}\ar@<12pt>[l]^-{\ p_2 \,}
A \coten B A \dots A^{\expcoten B n-1} 
\ar@<14pt>[r]|-{\,\sigma_0 \,} \ar@<-17pt>[r]|-{\,\sigma_{n-1} \,} &
\ar@<-21pt>[l]_-{\,\partial_0 \,} \ar@<-7pt>[l]|-{\,\partial_1 \,}^-\vdots
\ar@<24pt>[l]^-{\,\partial_n \,}
A^{\expcoten B n}  \dots}
$$
where for any positive integer $n$ we put
$$
\begin{array}{ll}
\sigma_k:= 1^{\diagcoten  n-k-1} \morcoten i \morcoten 1^{\diagcoten k}
& \textrm{ for } 0\leq k <n \\
\partial_0:= 1^{\diagcoten  n-1} \morcoten t=p_{1\dots n-1}& \\
\partial_k:= 1^{\diagcoten  n-k-1} \morcoten d \morcoten 1^{\diagcoten k-1}
& \textrm{ for } 0 < k <n \\
\partial_n:= s\morcoten 1^{\diagcoten  n-1} =p_{2\dots n}  .&\\
\end{array}
$$
By the functoriality of $\morcoten$ --- cf. \cite[Proposition 3.5]{Bohm:Xmod_I} --- they constitute a simplicial monoid which obeys property (b) by
construction and for which the morphisms $q_n$ of part (c) are invertible. In order to see that it has Moore length 1, note first that $A^{(1)}=A \coten B I$ exists; see \cite[Theorem 1.1]{Bohm:Xmod_II}. We claim that also for any $n>0$ there is an $\mathcal S$-relative pullback
$$
\xymatrix@C=15pt@R=15pt{
A\coten B I \ar[rr]^-{p_I} \ar[d]_-{p_A} &&
I \ar@/^1.1pc/[dd]^-u \ar[d]_-u \\
A \ar[d]_-{1\diagcoten i \diagcoten \cdots \diagcoten i} \ar[rr]^-s &&
B \ar[d]_-{i \diagcoten \cdots \diagcoten i} \\
A^{\expcoten B n} \ar[rr]_-{p_{2\dots n}} &&
A^{\expcoten B n-1}}
$$
determining $(A^{\expcoten B n})^{(1)}$. By construction 
$\xymatrix@C=12pt{A & A\coten B I \ar[r]^-{p_I} \ar[l]_-{p_A} & I}\in \mathcal S$ hence by (POST) also 
$\xymatrix@C=12pt{A^{\expcoten B n}  && \ar[ll]_-{1\diagcoten i \diagcoten \cdots \diagcoten i} A & A\coten B I \ar[r]^-{p_I} \ar[l]_-{p_A} & I}\in \mathcal S$.
If some morphisms $f$ and $g$ render commutative the first diagram of 
$$
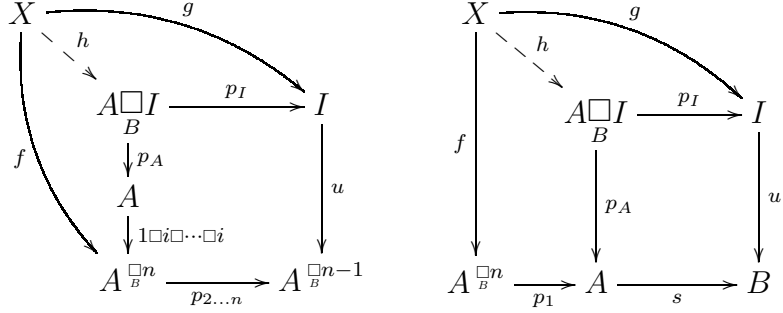
\xymatrix@C=15pt@R=15pt{
X \ar@{-->}[rd]^-h \ar@/^1.2pc/[rrrd]^-g \ar@/_1.2pc/[rddd]_-f \\
& A \coten B I \ar[rr]^-{p_I} \ar[d]^-{p_A} &&
I \ar[dd]^-u \\
& A \ar[d]^-{1\diagcoten i \diagcoten \cdots \diagcoten i} \\
& A^{\expcoten B n} \ar[rr]_-{p_{2\dots n}} &&
A^{\expcoten B n-1}}
\qquad
\xymatrix@C=15pt@R=18pt{
X \ar@{-->}[rd]^-h \ar@/^1.2pc/[rrrd]^-g \ar[ddd]_-f \\
& A \coten B I \ar[rr]^-{p_I} \ar[dd]^-{p_A} &&
I \ar[dd]^-u \\
\\
A^{\expcoten B n} \ar[r]_-{p_1} &
A \ar[rr]_-s &&
B} 
$$
then they make commute the second diagram as well by the commutativity of
$$
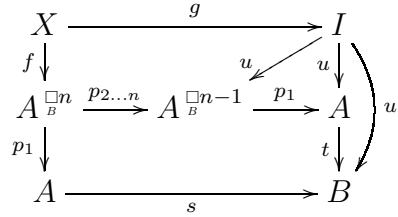
\xymatrix@R=15pt{
X \ar[rr]^-g \ar[d]_-f &&
I \ar[ld]_(.6)u \ar[d]_-u  \ar@/^1.1pc/[dd]^-u \\
A^{\expcoten B n} \ar[r]^-{p_{2\dots n}} \ar[d]_-{p_1}&
A^{\expcoten B n-1} \ar[r]^-{p_1} &
A \ar[d]_-t \\
A \ar[rr]_-s &&
B}
$$
Whenever $\xymatrix@C=12pt{A^{\expcoten B n} & \ar[l]_-f X \ar[r]^-g & I} \in \mathcal S$ also 
$\xymatrix@C=12pt{A & \ar[l]_-{p_1} A^{\expcoten B n} & \ar[l]_-f X \ar[r]^-g & I} \in \mathcal S$ by (POST). Hence by the universality of the $\mathcal S$-relative pullback in the second diagram, it has a unique filler $h$. But the same morphism $h$ is a filler also for the first diagram by the commutativity of both diagrams
$$
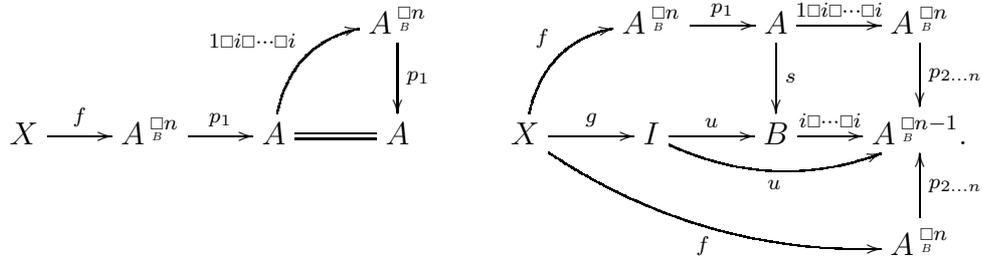
\xymatrix{
&&& A^{\expcoten B n} \ar[d]^-{p_1} \\
X \ar[r]^-f &
A^{\expcoten B n} \ar[r]^-{p_1} &
A \ar@{=}[r] \ar@/^1.2pc/[ru]^-{1\diagcoten i \diagcoten \cdots \diagcoten i}&
A} \qquad
\xymatrix{
& A^{\expcoten B n} \ar[r]^-{p_1} &
A \ar[r]^-{1\diagcoten i \diagcoten \cdots \diagcoten i} \ar[d]^-s &
A^{\expcoten B n} \ar[d]^-{p_{2\dots n}} \\
X \ar[r]^-g \ar@/^1.2pc/[ru]^-f \ar@/_1.2pc/[rrrd]_-f &
I \ar[r]^-u \ar@/_1.1pc/[rr]_-u &
B \ar[r]^-{i \diagcoten \cdots \diagcoten i}  &
A^{\expcoten B n-1} . \\
&&& A^{\expcoten B n} \ar[u]_-{p_{2\dots n}}}
$$
The unique filler of the second diagram is the unique filler for the first diagram since any filler $h$ of the first diagram is clearly a filler for the second diagram as well. 
For the reflection property let us use again that $1\morcoten i \morcoten \cdots \morcoten i$ is a monomorphism split by $p_1$. Then if
$\xymatrix@C=12pt{Y & \ar[l]_-f X \ar[r]^-g &A\coten B I \ar[r]^-{p_A} & A \ar[rr]^-{1\diagcoten i \diagcoten \cdots \diagcoten i} && A^{\expcoten B n}} \in \mathcal S$ then by (POST) also 
$\xymatrix@C=12pt{Y & \ar[l]_-f X \ar[r]^-g &A\coten B I \ar[r]^-{p_A} & A} \in \mathcal S$ . So whenever also 
$\xymatrix@C=12pt{Y & \ar[l]_-f X \ar[r]^-g &A\coten B I \ar[r]^-{p_I} & I} \in \mathcal S$, it follows from the reflectivity of 
$\xymatrix@C=12pt{A & \ar[l]_-{p_A} A\coten B I \ar[r]^-{p_I} & I}$ that 
$\xymatrix@C=12pt{Y & \ar[l]_-f X \ar[r]^-g &A\coten B I} \in \mathcal S$. Reflectivity on the left is proven symmetrically.

With this we proved that $(A^{\expcoten B n} )^{(1)}$ exists for any $n>0$ and
it is isomorphic to $A^{(1)}=A \coten B I$. 
The morphism $\partial_{n-1}^{(1)}$ is the identity morphism, being defined as
the unique morphism fitting the first commutative diagram of
$$
\xymatrix@C=15pt@R=15pt{
A \coten B I \ar@{-->}[rd]^(.55){\partial_{n-1}^{(1)}} \ar@/^1.2pc/[rrrd]^-{p_I}
\ar[dd]_-{p_A} \\
& A \coten B I  \ar[rr]^-{p_I} \ar[d]^-{p_A} &&
I \ar[dd]^-u \\
A \ar@{=}[r] \ar[d]_-{1\diagcoten i \diagcoten \cdots \diagcoten i} &
A \ar[d]_-{1\diagcoten i \diagcoten \cdots \diagcoten i} \\
A^{\expcoten B n} \ar[r]_-{d \diagcoten 1 \diagcoten \cdots \diagcoten 1} &
A^{\expcoten B n-1} \ar[rr]_-{p_{2\dots n-1}} &&
A^{\expcoten B n-2}}\qquad\qquad
\raisebox{17pt}{$\xymatrix@C=46pt@R=46pt{
\\
I \ar@{=}[r] \ar[d]_-u &
I \ar[d]^-u \\
A \coten B I \ar@{=}[r] &
A \coten B I .}$}
$$
Since 
$\xymatrix@C=12pt{A\coten B I & \ar[l]_-u I \ar@{=}[r] & I}\in \mathcal S$ by
the unitality of $\mathcal S$, the second diagram is obviously an $\mathcal
S$-relative pullback. This proves that for any $n>1$ 
$$
(A^{\expcoten B n} )^{(2)}=
(A^{\expcoten B n} )^{(1)} 
\raisebox{-5pt}{$\coten {(A^{\expcoten B n-1} )^{(1)}}$} I \cong
(A\coten B I)  
\raisebox{-7pt}{$\stackrel{\displaystyle \Box} {\scriptscriptstyle {A \expcoten B I}}$} I
$$ 
exists and it is isomorphic to $I$. Then by Corollary \ref{cor:Moore_length}
$S$ has Moore length 1.
Above we proved that 
$\xymatrix{(A^{\expcoten B n})^{(1)} \ar[r]^-{\partial_{n-1}^{(1)}} & 
(A^{\expcoten B n-1})^{(1)}}$ is invertible; then so is its inverse 
$\sigma_{n-2}^{(1)}$. Therefore by Proposition \ref{prop:l1_y} the morphism
$y_{(n,k)}$ of \eqref{eq:y_nk} is invertible for all $0<k<n$. For any $n>0$ the morphism
$y_{(n,0)}$ takes now the form 
$$
\xymatrix{
(A^{\expcoten B n})^{(1)}A^{\expcoten B n-1}\cong
(A\coten B I)A^{\expcoten B n-1} \ar[r]^-{p_A 1} &
A A^{\expcoten B n-1} 
\ar[rr]^-{(1\diagcoten i \diagcoten \cdots \diagcoten i)(i\diagcoten 1)} &&
(A^{\expcoten B n})^2 \ar[r]^-m &
A^{\expcoten B n}}
$$
in which we recognize the invertible morphism $q_n$.
This proves that the nerve of an object of 
$\mathsf{CatMon}_{\mathcal S}(\mathsf C)$ is indeed an object of
$\mathsf{Simp}^1\mathsf{Mon}_{\mathcal S}(\mathsf C)$. 
A morphism 
$$
\xymatrix@C=25pt{
(B \ar@{ >->}|(.55){\, i \,}[r] &
A \ar@{->>}@<-5pt>[l]_-s \ar@{->>}@<5pt>[l]^- t &
A \coten B A \ar[l]_-d) \ar[r]^-{(b,a)} &
(B' \ar@{ >->}|(.55){\, i' \,}[r] &
A' \ar@{->>}@<-5pt>[l]_-{s'} \ar@{->>}@<5pt>[l]^-{t'} &
A' \coten {B'} A' \ar[l]_-{d'})}
$$
is sent to the family 
$(\xymatrix@C=12pt{B \ar[r]^-b & B'},
\{\xymatrix@C=12pt{A^{\expcoten B n} \ar[r]^-{a^{\diagcoten n}} & 
A^{\prime \expcoten B n}}\}_{n>0})$
which is clearly a morphism of simplicial monoids.

It remains to see that the above constructed functors are mutually inverse
equivalences. Sending an $\mathcal S$-relative category 
(and $\mathcal S$-relative functor, respectively) to its nerve and then truncating it as above, we clearly re-obtain the $\mathcal S$-relative
category (and $\mathcal S$-relative functor, respectively) that we started
with. Iterating the functors in the opposite order, an object $S$ of
$\mathsf{Simp}^1\mathsf{Mon}_{\mathcal S}(\mathsf C)$ is sent to
\begin{equation} \label{eq:image}
\xymatrix@C=35pt{
S_0 \ar[r]|-{\,\sigma_0 \,} &
\ar@<-5pt>[l]_-{\partial_0}\ar@<5pt>[l]^-{\partial_1} S_1
\ar@<7pt>[r]|-{\,1\diagcoten \sigma_0 \,} \ar@<-7pt>[r]|-{\,\sigma_0\diagcoten 1 \,} &
\ar[l]|-{\, d \,}
\ar@<-12pt>[l]_-{\, p_1 \,}\ar@<12pt>[l]^-{\ p_2 \,}
S_1 \coten {S_0} S_1 \dots S_1^{\expcoten {S_0} n-1} 
\ar@<16pt>[r]|-{\,\widetilde \sigma_0 \,} \ar@<-17pt>[r]|-{\,\widetilde\sigma_{n-1} \,} &
\ar@<-21pt>[l]_-{\,\widetilde\partial_0 \,} \ar@<-7pt>[l]|-{\,\widetilde\partial_1 \,}^-\vdots
\ar@<24pt>[l]^-{\,\widetilde\partial_n \,}
S_1^{\expcoten {S_0} n}  \dots}
\end{equation}
where 
$d=\xymatrix@C=16pt{
S_1 \coten {S_0} S_1 \ar[r]^-{q_2^{-1}} &
(S_1 \coten {S_0} I)S_1 \ar[r]^-{p_{S_1}1} &
S_1^2 \ar[r]^-m &
S_1}$
(see \cite[Proposition 3.8]{Bohm:Xmod_II}) and for any positive integer $n$,
$$
\begin{array}{ll}
\widetilde\sigma_k:= 1^{\diagcoten  n-k-1} \morcoten \sigma_0 \morcoten 1^{\diagcoten k}
& \textrm{ for } 0\leq k <n \\
\widetilde\partial_0:= 1^{\diagcoten  n-1} \morcoten \partial_0=p_{1\dots n-1}& \\
\widetilde\partial_k:= 1^{\diagcoten  n-k-1} \morcoten d \morcoten 1^{\diagcoten k-1}
& \textrm{ for } 0 < k <n \\
\widetilde\partial_n:= \partial_1 \morcoten 1^{\diagcoten  n-1} =p_{2\dots n}  .&\\
\end{array}
$$
Note that together with the family of morphisms $\{q_n\}_{n> 0}$ they render
commutative the following diagrams. For all $0\leq i <n$,
\begin{equation} \label{eq:q_sigma}
\xymatrix@C=20pt{
S_1^{(1)} S_1^{\expcoten {S_0} n-1} \ar[r]_-{p_{S_1} \widetilde \sigma_{n-1}}
\ar[d]^-{1\widetilde \sigma_i} \ar@/^2pc/[rrrrrr]^-{q_n} &
S_1 S_1^{\expcoten {S_0} n}  \ar[r]_-{\widetilde \sigma_0 1}
\ar[d]_-{1\widetilde \sigma_i} &
\cdots \ar[r]_-{\widetilde \sigma_0 1} & 
S_1^{\expcoten {S_0} n-i} S_1^{\expcoten {S_0} n} 
\ar[d]^-{\widetilde \sigma_0\widetilde \sigma_i} 
\ar[r]_-{\widetilde \sigma_0 1} &
\cdots \ar[r]_-{\widetilde \sigma_0 1} & 
(S_1^{\expcoten {S_0} n})^2 \ar[r]_-m 
\ar[d]^-{\widetilde \sigma_i\widetilde \sigma_i} &
S_1^{\expcoten {S_0} n}\ar[d]_-{\widetilde \sigma_i} \\
S_1^{(1)} S_1^{\expcoten {S_0} n} \ar[r]^-{p_{S_1} \widetilde \sigma_{n}}
\ar@/_1.5pc/[rrrrrr]_-{q_{n+1}} &
S_1 S_1^{\expcoten {S_0} n+1}  \ar[r]^-{\widetilde \sigma_0 1} &
\cdots \ar[r]^-{\widetilde \sigma_0 1} & 
S_1^{\expcoten {S_0} n+1-i} S_1^{\expcoten {S_0} n+1} \ar[r]^-{\widetilde \sigma_0 1} &
\cdots \ar[r]^-{\widetilde \sigma_0 1} & 
(S_1^{\expcoten {S_0} n+1})^2 \ar[r]^-m &
S_1^{\expcoten {S_0} n+1} .}
\end{equation}
For $1< i <n$,
\begin{equation} \label{eq:q_delta}
\xymatrix@C=7pt@R=30pt{
S_1^{(1)} S_1^{\expcoten {S_0} n} \ar[r]_-{p_{S_1} \widetilde \sigma_{n}}
\ar[d]^-{1\widetilde \partial_i} \ar@/^2pc/[rrrrrrr]^-{q_{n+1}} &
S_1 S_1^{\expcoten {S_0} n+1}  \ar[r]_-{\widetilde \sigma_0 1}
\ar[d]_-{1\widetilde \partial_i} &
\cdots \ar[r]_-{\widetilde \sigma_0 1} & 
S_1^{\expcoten {S_0} n-i} S_1^{\expcoten {S_0} n+1} 
\ar[rd]_(.3){1\widetilde \partial_i} \ar[r]_-{\widetilde \sigma_0 1} &
S_1^{\expcoten {S_0} n+1-i} S_1^{\expcoten {S_0} n+1}
\ar[d]^-{\widetilde \partial_0\widetilde \partial_i} 
\ar[r]_-{\widetilde \sigma_0 1} &
\cdots \ar[r]_-{\widetilde \sigma_0 1} & 
\raisebox{-5pt}{$(S_1^{\expcoten {S_0} n+1})^2$} \ar[r]_-m 
\ar[d]^-{\widetilde \partial_i\widetilde \partial_i} &
S_1^{\expcoten {S_0} n+1}\ar[d]_-{\widetilde \partial_i} \\
S_1^{(1)} S_1^{\expcoten {S_0} n-1} 
\ar[r]^-{\raisebox{12pt}{${}_{p_{S_1} \widetilde \sigma_{n-1}}$}}
\ar@/_1.5pc/[rrrrrrr]_-{q_{n}} &
S_1 S_1^{\expcoten {S_0} n}  \ar[rr]^-{\widetilde \sigma_0 1} &&
\cdots \ar[r]^-{\widetilde \sigma_0 1} & 
S_1^{\expcoten {S_0} n-i} S_1^{\expcoten {S_0} n} \ar[r]^-{\widetilde \sigma_0 1} &
\cdots \ar[r]^-{\widetilde \sigma_0 1} & 
(S_1^{\expcoten {S_0} n})^2 \ar[r]^-m &
S_1^{\expcoten {S_0} n}}
\end{equation}
and for $0\leq i \leq 1<n$ the analogous one
\begin{equation} \label{eq:q_delta01}
\xymatrix@C=32pt@R=30pt{
S_1^{(1)} S_1^{\expcoten {S_0} n} \ar[r]_-{p_{S_1} \widetilde \sigma_{n}}
\ar[d]^-{1\widetilde \partial_i} \ar@/^2pc/[rrrrr]^-{q_{n+1}} &
S_1 S_1^{\expcoten {S_0} n+1}  \ar[r]_-{\widetilde \sigma_0 1}
\ar[d]_-{1\widetilde \partial_i} &
\cdots \ar[r]_-{\widetilde \sigma_0 1} & 
S_1^{\expcoten {S_0} n} S_1^{\expcoten {S_0} n+1} 
\ar[rd]_(.3){1\widetilde \partial_i} \ar[r]_-{\widetilde \sigma_0 1} &
\raisebox{-5pt}{$(S_1^{\expcoten {S_0} n+1})^2$} \ar[r]_-m 
\ar[d]^-{\widetilde \partial_i\widetilde \partial_i} &
S_1^{\expcoten {S_0} n+1}\ar[d]_-{\widetilde \partial_i} \\
S_1^{(1)} S_1^{\expcoten {S_0} n-1} 
\ar[r]^-{\raisebox{12pt}{${}_{p_{S_1} \widetilde \sigma_{n-1}}$}}
\ar@/_1.5pc/[rrrrr]_-{q_{n}} &
S_1 S_1^{\expcoten {S_0} n}  \ar[rr]^-{\widetilde \sigma_0 1} &&
\cdots \ar[r]^-{\widetilde \sigma_0 1} & 
(S_1^{\expcoten {S_0} n})^2 \ar[r]^-m &
S_1^{\expcoten {S_0} n}.}
\end{equation}

We claim that a natural isomorphism from $S$ to its image in \eqref{eq:image} can be constructed iteratively for all $n\geq 0$ as 
\begin{equation} \label{eq:wn}
\begin{array}{l}
\bullet \  \ w_0:=1 \\
\bullet \  \ w_n:= \xymatrix@C=20pt{
S_n\ar[r]^-{y_{(n,0)}^{-1}} & 
S_n^{(1)} S_{n-1} \ar[r]^-{\partial_0^{(1)}1} &
\cdots \ar[r]^-{\partial_0^{(1)}1} &
S_1^{(1)} S_{n-1}  \ar[r]^-{1w_{n-1}} &
S_1^{(1)} S_1^{\expcoten {S_0} n-1} \ar[r]^-{q_n} &
S_1^{\expcoten {S_0} n}}.
\end{array}
\end{equation}
Note that it gives $w_1=1$ but non-trivial higher components. Let us prove by induction in $n$ the equality
\begin{equation} \label{eq:p1.wn}
\xymatrix@C=15pt{
S_n \ar[r]^-{w_n} &
S_1^{\expcoten {S_0} n} \ar[r]^-{p_1} &
S_1}=
\xymatrix@C=15pt{
S_n \ar[r]^-{\partial_0} &
\cdots \ar[r]^-{\partial_0} &
S_1}
\qquad \textrm{for} \ n>0.
\end{equation}
For $n=1$ both sides yield the identity morphism thus the equality holds. If it holds for some $n>0$ then the following diagram commutes
$$
\xymatrix@R=15pt{
S_{n+1} \ar[r]_-{y_{(n+1,0)}^{-1}} \ar@/^2pc/[rrrrr]^-{w_{n+1}} \ar[d]_-{\partial_0} \ar@{}[rrrdd]|-{\eqref{eq:y_id}} & 
S_{n+1}^{(1)} S_{n} \ar[r]_-{\partial_0^{(1)}} &
\cdots \ar[r]_-{\partial_0^{(1)}} &
S_1^{(1)} S_{n-1}  \ar[r]_-{1w_{n}} \ar[d]_-{1 \partial_0} \ar@{}[rdd]|-{\textrm{(IH)}}&
S_1^{(1)} S_1^{\expcoten {S_0} n} \ar[r]_-{q_{n+1}} \ar[d]^-{1p_1} 
\ar@{}[rdd]|-{
\txt{\tiny \cite[(3.8)]{Bohm:Xmod_II}}
} &
S_1^{\expcoten {S_0} n+1} \ar[dd]^-{p_1}  \\
\vdots  \ar[d]_-{1 \partial_0} &&&
\vdots  \ar[d]_-{1 \partial_0} &
S_1^{(1)} S_1 \ar[d]^-{1 \partial_0} \\
S_1 \ar[rrr]^-{y_{(1,0)}^{-1}} \ar@/_1.5pc/@{=}[rrrrr] &&&
S_1^{(1)} S_0 \ar@{=}[r] &
S_1^{(1)} S_0 \ar[r]^-{q_1} &
S_1\\
&}
$$
proving \eqref{eq:p1.wn} for all positive $n$. (The region marked by (IH) commutes by the induction hypothesis.)

The morphisms $w_n$ of \eqref{eq:wn} are composites of isomorphisms (see Proposition \ref{prop:l1_y}) hence they are invertible. They are compatible with the units of the domain and codomain monoids by the unitality of the constituent monoid morphisms. Multiplicativity is checked by induction in $n$ again. Trivially, $w_0=1$ is multiplicative. If $w_{n-1}$ is multiplicative for some $n>0$ then by 
\cite[Corollary 1.7]{Bohm:Xmod_I} so are both isomorphisms
\begin{equation}\label{eq:w_factors}
 \xymatrix{
S_1^{(1)} S_{n-1} \ar[r]^-{\sigma_0^{(1)}1} &
\cdots \ar[r]^-{\sigma_0^{(1)}1} &
S_n^{(1)} S_{n-1} \ar[r]^-{y_{(n,0)}} &
S_n}
\ \  \textrm{and} \ \ 
 \xymatrix{
S_1^{(1)} S_{n-1} \ar[r]^-{1w_{n-1}} &
S_1^{(1)} S_1^{\expcoten {S_0} n-1} \ar[r]^-{q_n} &
S_1^{\expcoten {S_0} n}}
\end{equation}
with respect to the multiplications induced by the respective distributive laws 
$$
\scalebox{.99}{$
\xymatrix@R=10pt@C=5pt{
S_{n-1} S_1^{(1)} \ar[r]^-{1\sigma_0^{(1)} } &
\cdots \ar[r]^-{1\sigma_0^{(1)}} &
S_{n-1} S_n^{(1)} \ar[rrr]^-{\sigma_{n-1} p_{S_n}} &&&
S_n^2 \ar[r]^-m &
S_n \ar[r]^-{y_{(n,0)}^{-1}} &
S_n^{(1)} S_{n-1} \ar[r]^-{\partial_0^{(1)} 1} &
\cdots \ar[r]^-{\partial_0^{(1)}1} &
S_1^{(1)} S_{n-1} \\
S_{n-1} S_1^{(1)} \ar[r]^-{\raisebox{10pt}{${}_{w_{n-1}1}$}} &
S_1^{\expcoten {S_0} n-1}S_1^{(1)} \ar[r]^-{\raisebox{10pt}{${}_{1p_{S_1}}$}} &
S_1^{\expcoten {S_0} n-1}S_1 \ar[rrrr]^-{\raisebox{10pt}{${}_{(\sigma_0 \diagcoten 1)(1\diagcoten \sigma_0 \diagcoten \cdots \diagcoten \sigma_0)}$}} &&&&
(S_1^{\expcoten {S_0} n} )^2 \ar[r]^-{\raisebox{10pt}{${}_{m}$}} &
S_1^{\expcoten {S_0} n}  \ar[r]^-{\raisebox{10pt}{${}_{q_n^{-1}} $}} &
S_1^{(1)} S_1^{\expcoten {S_0} n-1} \ar[r]^-{\raisebox{13pt}{${}_{\ \ 1w_{n-1}^{-1}}$}} &
S_1^{(1)} S_{n-1}.}$}
$$
Their equality follows by the commutativity of the diagrams of Figure \ref{fig:equal_dlaws}, whose right verticals are joint monomorphisms.  
\begin{figure} 
\centering
\begin{sideways}
\scalebox{.95}{
$\xymatrix@C=5pt@R=15pt{
&&&&&& 
S_{n-1} S_n^{(1)} \ar[rr]^-{1p_{S_n}} \ar[d]^-{\partial_0\partial_0^{(1)}} &&
S_{n-1} S_n \ar[rr]^-{\sigma_{n-1}1} \ar[d]^-{\partial_0\partial_0} &&
S_n^2 \ar[rr]^-m \ar[d]^-{\partial_0\partial_0} &&
S_n \ar[rr]^-{w_n} \ar@{}[rrdd]|-{\eqref{eq:p1.wn}}
\ar[d]^-{\partial_0}
&&
S_1^{\expcoten {S_0} n} \ar[dd]^-{p_1} \\
&& 
&&&&
\vdots \ar[d]^-{\partial_0\partial_0^{(1)}} &&
\vdots \ar[d]^-{\partial_0\partial_0} &&
\vdots \ar[d]^-{\partial_0\partial_0} &&
\vdots \ar[d]^-{\partial_0} \\
S_{n-1}S_1^{(1)}
\ar@/^1.2pc/[rrrrrruu]|(.44)*=0@{>}
|-{\raisebox{4pt}{\rotatebox[origin=c]{15}{$\ \dots\ $}}}
^(.3){1\sigma_0^{(1)}}^(.7){1\sigma_0^{(1)}} 
\ar@/_1.2pc/[rrrrdd]_-{w_{n-1}1} 
\ar[rr]^-{\partial_0 1} &&
\cdots \ar[rr]^-{\partial_0 1} \ar@{}[rrd]_-{\eqref{eq:p1.wn}} &&
S_1S_1^{(1)} \ar[rr]^-{\partial_0 1} &&
S_0S_1^{(1)} \ar[rr]^-{1p_{S_1}} &&
S_0S_1 \ar[rr]^-{\sigma_0 1} &&
S_1^2 \ar[rr]^-m &&
S_1 \ar@{=}[rr] &&
S_1 \\
&&&&&&&& S_1^2 \ar[u]_-{\partial_0 1} \\
&&&& S_1^{\expcoten {S_0} n-1} S_1^{(1)} \ar[uu]_-{p_11} \ar[rrrr]_-{1p_{S_1}} &&&&
S_1^{\expcoten {S_0} n-1} \ar[u]_-{p_11} S_1
\ar[rr]_-{\raisebox{-10pt}{${}_{(\sigma_0 \diagcoten 1)
(1\diagcoten \sigma_0 \diagcoten \cdots \diagcoten \sigma_0)}$}} &&
(S_1^{\expcoten {S_0} n})^2 
\ar[rrrr]_-m
\ar[uu]_-{p_1p_1} 
&&&&
S_1^{\expcoten {S_0} n}\ar[uu]_-{p_1} \\
\\
\\
&& S_{n-1} S_n^{(1)} \ar[r]^-{\raisebox{10pt}{${}_{1p_{S_n}}$}} \ar[dd]^-{1p_I} &
S_{n-1} S_n \ar[r]^-{\sigma_{n-1}1} \ar[dd]^-{1\partial_n} &
S_n^2 \ar[r]^-m \ar[dd]^-{\partial_n\partial_n} &
S_n \ar[r]_-{y_{(n,0)}^{-1}} \ar[dd]_-{\partial_n} 
\ar@{}[rdd]|-{\eqref{eq:y_more}} \ar@/^2pc/[rrrrrrrrr]^-{w_n} &
\raisebox{-5pt}{$S_n^{(1)} S_{n-1}$} 
\ar[r]_-{\raisebox{-11pt}{${}_{\partial_0^{(1)} 1}$}} \ar[dd]^-{p_I1} &
\cdots \ar[r]_-{\partial_0^{(1)} 1} &
S_1^{(1)} S_{n-1} \ar[rr]_-{\raisebox{-10pt}{${}_{1w_{n-1}}$}} \ar[dd]^-{p_I 1} &&
\raisebox{0pt}{$S_1^{(1)} S_1^{\expcoten {S_0} n-1}$} \ar[rrrr]_-{q_n} 
\ar[dd]^-{p_I 1}  &&
\ar@{}[dd]|(.4){
\txt{\tiny \cite[(3.9)]{Bohm:Xmod_II}}}&&
S_1^{\expcoten {S_0} n} \ar[dd]^-{p_{2\dots n}} \\
\\
S_{n-1}S_1^{(1)}
\ar@/^1.2pc/[rruu]|(.3)*=0@{>}
|-{\rotatebox[origin=c]{85}{$\ddots$}\ \ }
^(.3){1\sigma_0^{(1)}} ^(.7){1\sigma_0^{(1)}} 
\ar@/_2pc/[rrrrrrrrrrdd]_-{w_{n-1}1} 
\ar[rr]^-{1p_I} &&
S_{n-1} \ar[r]^-{1u} \ar@/_1.1pc/@{=}[rrr] &
S_{n-1}^2 \ar@{=}[r] &
S_{n-1}^2 \ar[r]^-m &
S_{n-1} \ar@{=}[r] &
S_{n-1} \ar@{=}[rr] &&
S_{n-1} \ar[rr]^-{w_{n-1}} &&
S_1^{\expcoten {S_0} n-1} \ar[r]_-{1u} \ar@/^2pc/@{=}[rrrr] &
S_1^{\expcoten {S_0} n-1} S_0 
\ar[rr]_-{\raisebox{-10pt}{${}_{1(\sigma_0 \diagcoten \cdots \diagcoten \sigma_0)}$}} &&
(S_1^{\expcoten {S_0} n-1})^2 \ar[r]_-m &
S_1^{\expcoten {S_0} n-1} \\
\\
&&&&&&&&&& S_1^{\expcoten {S_0} n-1}S_1^{(1)} \ar[r]_-{1p_{S_1}}
\ar[uu]^-{1p_I} &
S_1^{\expcoten {S_0} n-1}S_1 \ar[uu]_-{1\partial_1}
\ar[rr]_-{\raisebox{-10pt}{${}_{(\sigma_0 \diagcoten 1)
(1\diagcoten \sigma_0 \diagcoten \cdots \diagcoten \sigma_0)}$}} &&
(S_1^{\expcoten {S_0} n})^2 \ar[r]_-m \ar[uu]^-{p_{2\dots n}p_{2\dots n}} &
S_1^{\expcoten {S_0} n} \ar[uu]_-{p_{2\dots n}}} $}
\end{sideways}
\caption{Multiplicativity of $w_n$}
\label{fig:equal_dlaws}
\end{figure}
Since $w_n$ is the composite of the second morphism of \eqref{eq:w_factors}
with the inverse of the first one, we conclude that it is multiplicative too.

Next we check by induction in $n$ that $\{w_n \}_{n\geq 0}$ is a simplicial
morphism; that is, 
\begin{equation}\label{eq:w_simp}
\xymatrix{
S_n \ar[r]^-{w_n} \ar[d]_-{\sigma_i} &
S_1^{\expcoten {S_0} n} \ar[d]^-{\widetilde \sigma_i} \\
S_{n+1} \ar[r]_-{w_{n+1}} &
S_1^{\expcoten {S_0} n+1}} \qquad
\xymatrix{
S_{n+1} \ar[r]^-{w_{n+1}} \ar[d]_-{\partial_j} &
S_1^{\expcoten {S_0} n+1} \ar[d]^-{\widetilde \partial_j} \\
S_{n} \ar[r]_-{w_{n}} &
S_1^{\expcoten {S_0} n}}
\end{equation}
commute for all $n\geq 0$ and $0\leq i \leq n$ and $0\leq j \leq n+1$.
Note that the induction must be started with $n=1$ because
the first diagram of \eqref{eq:w_delta_penult} below only makes sense for
$n>0$.
For $n=0$ the diagrams of \eqref{eq:w_simp} commute because $w_0$ and $w_1$
are the identity morphisms and the equalities 
$\xymatrix@C=12pt{S_0 \ar[r]^-{\widetilde \sigma_0} & S_1}=
\xymatrix@C=12pt{S_0 \ar[r]^-{\sigma_0} & S_1}$ and 
$\xymatrix@C=12pt{S_1 \ar[r]^-{\widetilde \partial_i} & S_0}=
\xymatrix@C=12pt{S_1 \ar[r]^-{\partial_i} & S_0}$
hold for $i=0,1$ by construction.
For $n=1$ commutativity of the diagrams of \eqref{eq:w_simp} is checked as
follows.  
$$
\xymatrix{
S_1 \ar[r]_-{y_{(1,0)}^{-1}} \ar[d]_-{\sigma_0} \ar@/^1.2pc/@{=}[rrr] 
\ar@{}[rd]|-{\eqref{eq:y_id}} &
S_1^{(1)}S_0 \ar@{=}[r] \ar[d]^-{\sigma_0^{(1)}\sigma_0} &
S_1^{(1)}S_0\ar[r]_-{q_1} \ar[d]^-{1\sigma_0} \ar@{}[rd]|-{\eqref{eq:q_sigma}}&
S_1 \ar[d]^-{\widetilde \sigma_0} \\
S_2 \ar[r]^-{y_{(2,0)}^{-1}} \ar@/_1.2pc/[rrr]_-{w_2} &
S_2^{(1)}S_1 \ar[r]^-{\partial_0^{(1)}1} &
S_1^{(1)}S_1 \ar[r]^-{q_2} &
S_1 \coten {S_0} S_1}\quad
\xymatrix{
S_1 \ar@{=}[rrr] \ar[rd]^-{u1} \ar[d]_-{\sigma_1} &&&
S_1 \ar[d]^-{\widetilde \sigma_1} \ar[ld]_-{u1} \\
S_2 \ar[r]^-{y_{(2,0)}^{-1}} \ar@/_1.2pc/[rrr]_-{w_2} &
S_2^{(1)}S_1 \ar[r]^-{\partial_0^{(1)}1} &
S_1^{(1)}S_1 \ar[r]^-{q_2} &
S_1 \coten {S_0} S_1}
$$
$$
\xymatrix@C=10pt@R=72pt{
S_2 \ar[r]^-{w_2} \ar[d]^(.35){\partial_0} \ar@{}[rd]|-{\eqref{eq:p1.wn}}&
S_1 \coten {S_0} S_1 \ar[d]_(.65){\widetilde \partial_0=p_1} \\
S_1 \ar@{=}[r] &
S_1}\ \ \ 
\xymatrix@C=10pt@R=10pt{
S_2 \ar[r]_-{y_{(2,0)}^{-1}} \ar[ddd]^(.35){\partial_1}
\ar@{}[rrrddd]|-{\eqref{eq:y_more}} \ar@/^1.2pc/[rrr]^-{w_2} &
S_2^{(1)}S_1 \ar[r]_-{\raisebox{-10pt}{${}_{\partial_0^{(1)}1=\partial_1^{(1)}1}$}} &
S_1^{(1)}S_1 \ar[r]_-{q_2} \ar@/_1.5pc/@{=}[rd] &
S_1 \coten {S_0} S_1 \ar[d]_-{q_2^{-1}}
\ar@/^1.1pc/[ddd]^(.65){\widetilde \partial_1} \\
&&& S_1^{(1)}S_1\ \ \ar[d]_-{p_{S_1}1} \\
&&& S_1^2 \ar[d]_-m \\
S_1 \ar@{=}[rrr] &&&
S_1 } \ \ \ 
\xymatrix@C=10pt@R=72pt{
S_2 \ar[r]_-{y_{(2,0)}^{-1}} \ar[d]^(.35){\partial_2}
\ar@/^1.2pc/[rrr]^-{w_2} &
S_2^{(1)}S_1 \ar[r]_-{\partial_0^{(1)}1} \ar[d]^-{p_I 1} \ar@{}[ld]|-{\eqref{eq:y_more}}  &
S_1^{(1)}S_1 \ar[r]_-{q_2} \ar[d]_-{p_I1} 
\ar@{}[rd]|-{\txt{\tiny \cite[(3.9)]{Bohm:Xmod_II}}}  &
S_1 \coten {S_0} S_1 \ar[d]_(.65){\widetilde \partial_2=p_2} \\
S_1 \ar@{=}[r] &
S_1 \ar@{=}[r] &
S_1 \ar@{=}[r] &
S_1 }
$$
Now assume that the first diagram of \eqref{eq:w_simp} commutes for some $n>0$ and all $0\leq i \leq n$. By the commutativity of the diagrams
$$
\xymatrix@C=13pt{
S_{n+1} \ar[rr]_(.2){y_{(n+1,0)}^{-1}} \ar[d]_-{\sigma_i} \ar@/^1.4pc/[rrrrrr]^-{w_{n+1}} \ar@{}[rd]|-{\eqref{eq:y_id}}&&
S_{n+1}^{(1)} S_n \ar[r]_-{\partial_0^{(1)}1} \ar[d]^-{1\sigma_i} \ar@/_1.2pc/[ld]_(.7){\sigma_i^{(1)} \sigma_i} &
\cdots  \ar[r]_-{\partial_0^{(1)}1}  &
S_1^{(1)} S_n \ar[r]_-{1w_n} \ar@{}[rd]|(.5){\textrm{(IH)}} \ar[d]^-{1\sigma_i} &
\raisebox{-5pt}{$S_{1}^{(1)} S_1^{\expcoten {S_0} n}$} \ar[r]_-{q_{n+1}}  \ar[d]^-{1\widetilde \sigma_i} 
\ar@{}[rd]|-{\eqref{eq:q_sigma}} &
S_1^{\expcoten {S_0} n+1} \ar[d]^-{\widetilde \sigma_i} \\
S_{n+2} \ar[r]^(.33){y_{(n+2,0)}^{-1}}  \ar@/_1.4pc/[rrrrrr]_-{w_{n+2}} &
S_{n+2}^{(1)} S_{n+1}  \ar[r]^(.48){\raisebox{12pt}{${}_{\partial_0^{(1)}1=\partial_i^{(1)}1}$}} &
S_{n+1}^{(1)} S_{n+1} \ar[r]^-{\partial_0^{(1)}1}  &
\cdots  \ar[r]^-{\partial_0^{(1)}1}  &
S_{1}^{(1)} S_{n+1} \ar[r]^-{1w_{n+1}}  &
S_{1}^{(1)} S_1^{\expcoten {S_0} n+1} \ar[r]^-{q_{n+2}} &
S_1^{\expcoten {S_0} n+2} \\
S_{n+1} \ar@{=}[rrr] \ar[d]_-{\sigma_{n+1}}  \ar[rd]^-{u1} &&&
S_{n+1} \ar[rrr]^-{w_{n+1}} \ar[d]^-{u1} &&&
S_1^{\expcoten {S_0} n+1} \ar[d]^-{\widetilde \sigma_{n+1}} \ar[ld]_-{u1} \\
S_{n+2} \ar[r]^(.33){y_{(n+2,0)}^{-1}}  \ar@/_1.4pc/[rrrrrr]_-{w_{n+2}} &
S_{n+2}^{(1)} S_{n+1}  \ar[r]^-{\partial_0^{(1)}1} &
\cdots  \ar[r]^-{\partial_0^{(1)}1}  &
S_{1}^{(1)} S_{n+1} \ar[rr]^-{1w_{n+1}}  &&
S_{1}^{(1)} S_1^{\expcoten {S_0} n+1} \ar[r]^-{q_{n+2}} &
S_1^{\expcoten {S_0} n+2} }
$$
(where the region marked by (IH) commutes by the induction hypothesis)
we conclude that the first diagram of \eqref{eq:w_simp} commutes for all $n>0$ and all $0\leq i \leq n$. 

Assume next that the second diagram of \eqref{eq:w_simp} commutes for some $n>0$ and all $0\leq j \leq n+1$. Then the following diagrams commute for all $0\leq j \leq n$.
\begin{equation}\label{eq:w_delta}
\xymatrix@C=19pt{
S_{n+2} \ar[r]_-{y_{(n+2,0)}^{-1}} \ar[d]_-{\partial_{j}} \ar@/^1.4pc/[rrrrrr]^(.65){w_{n+2}} \ar@{}[rrd]|-{\eqref{eq:y_id}} &
\raisebox{-5pt}{$S_{n+2}^{(1)} S_{n+1}$} \ar[r]_(.55){\raisebox{-12pt}{${}_{\partial_0^{(1)}1=\partial_j^{(1)}1}$}}   &
S_{n+1}^{(1)} S_{n+1} \ar[r]_-{\partial_0^{(1)}1}   \ar[d]^-{1\partial_j} &
\cdots  \ar[r]_-{\partial_0^{(1)}1}  &
S_1^{(1)} S_{n+1} \ar[r]_-{\raisebox{-10pt}{${}_{1w_{n+1}}$}} \ar[d]^-{1\partial_j} \ar@{}[rd]|(.52){\textrm{(IH)}}&
\raisebox{-7pt}{$S_{1}^{(1)} S_1^{\expcoten {S_0} n+1}$} \ar[r]_-{q_{n+2}}  \ar[d]^-{1\widetilde \partial_j} 
\ar@{}[rd]|-{\eqref{eq:q_delta}}  &
S_1^{\expcoten {S_0} n+2} \ar[d]^-{\widetilde \partial_{j}} \\
S_{n+1} \ar[rr]^-{y_{(n+1,0)}^{-1}}  \ar@/_1.4pc/[rrrrrr]_(.35){w_{n+1}} &&
S_{n+1}^{(1)} S_{n}  \ar[r]^-{\partial_0^{(1)}1} &
\cdots  \ar[r]^-{\partial_0^{(1)}1}  &
S_{1}^{(1)} S_{n} \ar[r]^-{1w_{n}}  &
S_{1}^{(1)} S_1^{\expcoten {S_0} n} \ar[r]^-{q_{n+1}} &
S_1^{\expcoten {S_0} n+1}}
\end{equation}
\begin{equation}\label{eq:w_delta_last}
\xymatrix@C=26pt{
S_{n+2} \ar[rr]_-{y_{(n+2,0)}^{-1}} \ar[d]_-{\partial_{n+2}} \ar@/^1.4pc/[rrrrrr]^(.65){w_{n+2}} \ar@{}[rrd]|-{\eqref{eq:y_more}} &&
S_{n+2}^{(1)} S_{n+1} \ar[r]_-{\partial_0^{(1)}1} \ar[d]^-{p_I1}  &
\cdots  \ar[r]_-{\partial_0^{(1)}1}  &
S_1^{(1)} S_{n+1} \ar[r]_-{\raisebox{-10pt}{${}_{1w_{n+1}}$}} \ar[d]^-{p_I1} &
\raisebox{-7pt}{$S_{1}^{(1)} S_1^{\expcoten {S_0} n+1}$} \ar[r]_-{q_{n+2}}  \ar[d]^-{p_I1} 
\ar@{}[rd]|-{
\txt{\tiny \cite[(3.9)]{Bohm:Xmod_II}}}  &
S_1^{\expcoten {S_0} n+2} \ar[d]_(.3){p_{2\dots n+2}}^(.3){\widetilde \partial_{n+2}} |(.3)= \\
S_{n+1} \ar@{=}[rr] &&
S_{n+1} \ar@{=}[rr] &&
S_{n+1} \ar[r]_-{w_{n+1}} &
S_1^{\expcoten {S_0} n+1} \ar@{=}[r] &
S_1^{\expcoten {S_0} n+1}  }
\end{equation}
The missing case $j=n+1$ follows by the commutativity of the following diagrams whose vertical arrows are joint monomorphisms.
\begin{equation}\label{eq:w_delta_penult}
\xymatrix@C=15pt@R=20pt{
&& 
S_1^{\expcoten {S_0} n+2} \ar[rr]^-{\widetilde \partial_{n+1}} \ar[d]^-{\widetilde \partial_0} &&
S_1^{\expcoten {S_0} n+1} \ar[d]_-{\widetilde \partial_0} \ar@/^1.2pc/[ddd]^-{p_1} \\
&& S_1^{\expcoten {S_0} n+1} \ar[rr]^-{\widetilde \partial_{n}} \ar@{}[dd]|-{\textrm{(IH)}} &&
 S_1^{\expcoten {S_0} n} \ar[d]|(.37)*+<2pt>{{}_{\widetilde \partial_0}} 
\ar@/_1.1pc/[dd]_-{p_1} \\
&&& \ar@{}[d]|-{\eqref{eq:p1.wn}} &
\vdots \ar[d]|-*+<5pt>{{}_{\widetilde \partial_0}} \\
S_{n+2} \ar@/^1.5pc/[rruuu]^-{w_{n+2}} \ar@/_1.2pc/[rrdd]_-{\partial_{n+1}} \ar[r]^-{\partial_0} 
\ar@{}[rruuu]|-{\eqref{eq:w_delta}}&
S_{n+1} \ar[r]^-{\partial_{n}} \ar@/^.9pc/[ruu]^-{w_{n+1}} &
S_n \ar[r]^-{\partial_0} \ar@/^1.3pc/[rruu]^-{w_n} \ar@{}[rrdd]|-{\eqref{eq:p1.wn}} &
\cdots \ar[r]^-{\partial_0} &
S_1 \\
\\
&& S_{n+1} \ar[uu]_-{\partial_0} \ar[rr]_-{w_{n+1}} &&
S_1^{\expcoten {S_0} n+1} \ar[uu]_-{p_1} }
\xymatrix@C=15pt@R=20pt{
&& 
S_1^{\expcoten {S_0} n+2} \ar[r]^-{\widetilde \partial_{n+1}} 
\ar[dd]^-{\widetilde \partial_{n+2}} &
S_1^{\expcoten {S_0} n+1} 
\ar[ddd]_-{\widetilde \partial_{n+1}}^-{p_{2\dots n+1}}|-*+<5pt>{{}_=} \\
\\
&& S_1^{\expcoten {S_0} n+1} \ar@/^.8pc/[rd]^(.4){\widetilde \partial_{n+1}}
\ar@{}[dd]|-{\textrm{(IH)}} \\
S_{n+2} \ar@/^1.5pc/[rruuu]^-{w_{n+2}} \ar@/_1.2pc/[rrdd]_-{\partial_{n+1}} 
\ar[r]^-{\partial_{n+2}} \ar@{}[rruuu]|-{\eqref{eq:w_delta_last}} &
S_{n+1} \ar@/_.8pc/[rd]_-{\partial_{n+1}} \ar@/^.8pc/[ru]^-{w_{n+1}} &&
S_1^{\expcoten {S_0} n} \\
&& S_n \ar@/_.8pc/[ru]_-{w_n} \ar@{}[rd]|-{\textrm{(IH)}} \\
&& S_{n+1} \ar[u]_(.35){\partial_{n+1}} \ar[r]_-{w_{n+1}} &
S_1^{\expcoten {S_0} n+1} 
\ar[uu]_(.4){p_{2\dots n+1}}|(.4)*+<5pt>{{}_=}^(.4){\widetilde \partial_{n+1}} }
\end{equation}
This proves that the second diagram of \eqref{eq:w_simp} commutes for all
stated values of $n$ and $j$ and thus $w$ is a simplicial morphism. 

Finally, naturality of $w$ is proven by induction in $n$. 
For any simplicial monoid morphism 
$\{\xymatrix@C=12pt{S_n \ar[r]^-{f_n} & S'_n}\}_{n\geq 0}$ 
and $n=1$ the first diagram of 
$$
\xymatrix@R=26pt{
S_n \ar[r]^-{w_n} \ar[d]_-{f_n} &
S_1^{\expcoten {S_0} n} \ar[d]^-{f_1^{\diagcoten n}} \\
S'_n \ar[r]_-{w'_n} &
S_1^{\prime \expcoten {S'_0} n}}
\xymatrix@C=20pt{
S_{n+1} \ar[r]_-{y_{(n+1,0)}^{-1}} \ar[d]_-{f_{n+1}} \ar@{}[rd]|-{\eqref{eq:y_nat}}
\ar@/^1.2pc/[rrrrr]^-{w_{n+1}} &
\raisebox{-5pt}{$S_{n+1}^{(1)} S_n \ar[r]_-{\partial_0^{(1)}1}$} 
\ar[d]^-{f_{n+1}^{(1)}f_n} &
\cdots \ar[r]_-{\partial_0^{(1)}1} &
S_{1}^{(1)} S_n \ar[r]_-{1w_n}\ar[d]_-{f_{1}^{(1)}f_n} \ar@{}[rd]|-{\textrm{(IH)}} &
\raisebox{-5pt}{$S_{1}^{(1)} S_1^{\expcoten {S_0} n}$} \ar[r]_-{q_{n+1}} 
\ar[d]^-{f_{1}^{(1)}f_1^{\diagcoten n}} &
S_1^{\expcoten {S_0} n+1} \ar[d]^-{f_1^{\diagcoten n+1}} \\
S'_{n+1} \ar[r]^-{y_{(n+1,0)}^{\prime -1}} \ar@/_1.2pc/[rrrrr]_-{w'_{n+1}} &
S_{n+1}^{\prime (1)} S'_n \ar[r]^-{\partial_0^{\prime (1)}1} &
\cdots \ar[r]^-{\partial_0^{\prime (1)}1} &
S_{1}^{\prime (1)} S'_n \ar[r]^-{1w'_n} &
S_{1}^{\prime (1)} S_1^{\prime \expcoten {S_0} n} \ar[r]^-{q'_{n+1}} &
S_1^{\prime \expcoten {S_0} n+1} }
$$
evidently commutes.
If the first diagram commutes for some non-negative integer $n$ then so does
the second one too.
\end{proof}

The functor $\mathsf{Simp}^1\mathsf{Mon}_{\mathcal S}(\mathsf C) \to \mathsf{CatMon}_{\mathcal S}(\mathsf C)$ in the above proof can be composed with the functor $\mathsf{CatMon}_{\mathcal S}(\mathsf C) \to \mathsf{Xmod}_{\mathcal S}(\mathsf C)$ in the proof of \cite[Theorem 3.10]{Bohm:Xmod_II}. The resulting equivalence functor sends an object $S$ to 
$$
(S_0,S_1^{(1)},
\xymatrix@C=12pt{S_1^{(1)} \ar[r]^-{p_I} &I},
\xymatrix@C=12pt{S_1^{(1)} \ar[r]^-{D_0} & S_0},
\xymatrix@C=20pt{S_0S_1^{(1)} \ar[r]^-{\sigma_0p_{S_1}} & S_1^2 \ar[r]^-m & S_1 \ar[r]^-{q_1^{-1}} & S_1^{(1)} S_0}),
$$
where $D_0$ is the same morphism in Proposition \ref{prop:Moore}.

\begin{example} \label{ex:Simp_groupoid}
As in \cite[Example 1.2]{Bohm:Xmod_II}, take the (evidently admissible and monoidal) class  of all spans in the category $\mathsf C$ of spans over a given set $X$.
The equivalent categories of \cite[Example 3.11]{Bohm:Xmod_II} are equivalent
also to the following category.
\begin{itemize}
\item[{$\mathsf{Simp}$}]\hspace{-.3cm} $^1\mathsf{Mon}(\mathsf C)$ 
whose \\ 
\underline{objects} are simplicial categories $S$ such that the object set of $S_n$ for each $n\geq 0$ is the given set $X$ and the following conditions hold.
\begin{itemize}
\item[{(a)}] The Moore complex of $S$ has length 1.
\item[{(c)}] The morphisms $q_n$ and $y_{(n,k)}$ of Theorem \ref{thm:SCat_vs_Simp} are
invertible for all $0\leq k<n$ (equivalently, $q_1$, $y_{(n,0)}$ and $y_{(n,1)}$ are
invertible for all $0<n$). 
\end{itemize}
(There is no condition (b) because we are working with the class of all spans.)
\underline{morphisms} are the simplicial functors.
\end{itemize}
This category contains as a full subcategory the category of simplicial
groupoids of Moore length 1; which is therefore equivalent to the category of
internal categories in the category of groupoids; and also to the category of
crossed modules of groupoids, see \cite[Example 3.11]{Bohm:Xmod_II}.
\end{example}

\begin{example} \label{ex:Simp_CoMon}
Let $\mathsf M$ be a symmetric monoidal category in which equalizers exist and are preserved by taking the monoidal product with any object.
Take $\mathsf C$ to be the category of comonoids in $\mathsf M$ with the monoidal admissible class $\mathcal S$ in \cite[Example 2.3]{Bohm:Xmod_I} of spans in $\mathsf C$.
The equivalent categories of \cite[Example 3.12]{Bohm:Xmod_II} are also equivalent
to the category 
\begin{itemize}
\item[{$\mathsf{Simp}$}]\hspace{-.3cm} $^1\mathsf{Mon}_{\mathcal S}(\mathsf C)$ 
whose \\ 
\underline{objects} are simplicial bimonoids $S$ in $\mathsf M$ such that 
\begin{itemize}
\item[{(a)}] For all $n>0$ and for
$\widehat \partial_i:=\xymatrix@C=10pt{
S_n\ar[r]^-\delta &
S_n^2 \ar[r]^-{\delta 1} &
S_n^3 \ar[rr]^-{1\partial_i1} &&
S_nS_{n-1}S_n}$ for $0<i\leq n$,
\begin{equation}\label{eq:Moore1_CoMon}
\xymatrix@C=25pt{
I \ar[r]^-u &
S_n \ar@/^2.1pc/[rrr]^-{\widehat \partial_n}
\ar@/^1.4pc/[rrr]|-{\,\widehat\partial_{n-1}\,}_-\vdots
\ar@/_.7pc/[rrr]|-{\,\widehat\partial_{1}\,}
\ar@/_1.4pc/[rrr]_-{1u1.\delta} &&&
S_nS_{n-1}S_n}
\end{equation}
is a joint equalizer in $\mathsf M$ (that is; the Moore complex of $S$ has length 1, see Examples \ref{ex:ass_Moore_CoMon} and \ref{ex:zero_CoMon}). 
\item[{(b)}] $\partial_0 1.\delta=\partial_0 1.c.\delta$ and 
$\partial_1 1.\delta=\partial_1 1.c.\delta$.
\item[{(c)}] The morphisms $q_n$ and $y_{(n,k)}$ of Theorem
  \ref{thm:SCat_vs_Simp} are invertible for $0\leq k<n$ (equivalently, $q_1$, $y_{(n,0)}$ and $y_{(n,1)}$ are invertible for all $0<n$).
\end{itemize}
\underline{morphisms} are the simplicial bimonoid morphisms.
\end{itemize}
\end{example}

As a simple consequence we obtain the following result in \cite{Emir}.

\begin{proposition}
Let $\mathsf M$ be a symmetric monoidal category in which equalizers exist and are preserved by taking the monoidal product with any object.
Take $\mathsf C$ to be the category of comonoids in $\mathsf M$ with the monoidal admissible class $\mathcal S$ in \cite[Example 2.3]{Bohm:Xmod_I} of spans in $\mathsf C$.
The equivalent categories in \cite[Example 3.12]{Bohm:Xmod_II} and in Example
\ref{ex:Simp_CoMon} have equivalent full subcategories as follows.
\begin{itemize}
\item The full subcategory of $\mathsf{CatMon}_{\mathcal S}(\mathsf C)$ for
  whose objects  
$\xymatrix@C=20pt{
B \ar@{ >->}[r]|(.55){\, i\, } &
A \ar@{->>}@<-4pt>[l]_-s  \ar@{->>}@<4pt>[l]^-t &
A\coten B A \ar[l]_-d}$ 
both $A$ and $B$ are cocommutative Hopf monoids in $\mathsf M$.
\item The full subcategory of $\mathsf{Xmod}_{\mathcal S}(\mathsf C)$ for
whose objects 
$(B,Y,\xymatrix@C=12pt{BY \ar[r]^-l & Y},\xymatrix@C=12pt{Y \ar[r]^-k & B})$
both $Y$ and $B$ are cocommutative Hopf monoids in $\mathsf M$.
\item The full subcategory of $\mathsf{ReflGraphMon}_{\mathcal S}(\mathsf C)$
  for whose objects  
$\xymatrix@C=20pt{
B \ar@{ >->}[r]|(.55){\, i\, } &
A \ar@{->>}@<-4pt>[l]_-s  \ar@{->>}@<4pt>[l]^-t}$ 
the following conditions hold.
\begin{itemize}
\item $A$ and $B$ are cocommutative Hopf monoids (with antipodes $z$)
\item for the morphisms 
$$
\overrightarrow s:=\xymatrix@C=10pt{
A \ar[r]^-\delta & A^2 \ar[r]^-{1s} & AB \ar[r]^-{1z} & AB \ar[r]^-{1i} & A^2 \ar[r]^-m & A}, \ 
\overleftarrow t:=\xymatrix@C=10pt{
A \ar[r]^-\delta & A^2 \ar[r]^-{t1} & BA \ar[r]^-{z1} & BA \ar[r]^-{i1} & A^2 \ar[r]^-m & A}
$$
the following diagram commutes.
\begin{equation} \label{eq:vil_d}
\xymatrix{
A^2 \ar[r]^-{\overrightarrow s \overleftarrow t}   \ar[d]_-{\overrightarrow s \overleftarrow t} &
A^2 \ar[r]^-c &
A^2 \ar[d]^-m \\
A^2 \ar[rr]_-m &&
A}
\end{equation}
\end{itemize}
\item The full subcategory of 
$\mathsf{Simp}^1\mathsf{Mon}_{\mathcal S}(\mathsf C)$ for whose objects $S$
  the following conditions hold. 
\begin{itemize}
\item $S_n$ is a cocommutative Hopf monoid in $\mathsf M$ for all $n\geq 0$. 
\item The Moore complex of $S$ has length 1; that is,  \eqref{eq:Moore1_CoMon} is a joint
  equalizer in $\mathsf M$ for all $n>1$. 
\end{itemize}
\end{itemize}
\end{proposition}

\begin{proof}
We need to show that the category listed last is a subcategory of the category in Example
\ref{ex:Simp_CoMon}. Condition (b) of Example \ref{ex:Simp_CoMon}
becomes trivial thanks to the cocommutativity assumption. Concerning condition
(c), the morphisms $q_n$ are invertible by \cite[Proposition 3.13]{Bohm:Xmod_II}
and the morphisms $y_{(n,k)}$ are invertible by Example \ref{ex:y_Hopf}. 
\end{proof}


\bibliographystyle{plain}

\end{document}